\documentclass{birkjour}
\usepackage{mathrsfs}

\newtheorem{theorem}{Theorem}[section]

\newtheorem{lemma}[theorem]{Lemma}

\theoremstyle{definition}
\newtheorem{definition}[theorem]{Definition}

\numberwithin{equation}{section} 

\newcommand{\Rot}{\operatorname{\mathbf{curl}}}
\newcommand{\Div}{\operatorname{\mathrm{div}}}

\newcommand{\rot}{\operatorname{\mathrm{curl}}}

\newcommand{\loc}{\mathrm{loc}}

\newcommand  {\C}{{\mathbb C}}
\newcommand  {\N}{{\mathbb N}}

\newcommand  {\R}{{\mathbb R}}

\newcommand {\Id}{\mathrm {I}}

\newcommand  {\D}{\mathsf D}
 
\renewcommand{\AA}{\boldsymbol{\mathsf A}}

\renewcommand{\SS}{\boldsymbol{\mathsf S}}
\newcommand  {\TT}{\boldsymbol{\mathsf T}}

\newcommand  {\LL}{\boldsymbol{\mathsf L}}
\newcommand  {\HH}{\boldsymbol{\mathsf H}}
\newcommand  {\EE}{\boldsymbol{\mathsf E}}

\newcommand  {\Pp}{\boldsymbol{\mathsf P}} 

\newcommand {\Tr}{\operatorname{Trace}}

\newcommand  {\hh}{{\boldsymbol{ h}}}

\newcommand  {\nn}{\boldsymbol{ n}}

\newcommand  {\uu}{\boldsymbol{ u}}
\newcommand  {\jj}{\boldsymbol{ j}}
\newcommand  {\mm}{\boldsymbol{ m}}
\newcommand  {\vv}{\boldsymbol{ v}}

\newcommand{\transposee}[1]{{#1}^{\mathsf T}}    
\usepackage{color}
\newcommand{\Bk}{\color{black}}
\newcommand{\Rd}{\color{red}} 
\let\Rd\Bk      
   
\begin{document}

\title[Shape derivatives II: Dielectric scattering]{Shape derivatives of boundary integral operators in electromagnetic scattering. Part II: Application to  scattering by a homogeneous dielectric obstacle }
\author{Martin Costabel}
  \address{IRMAR, Institut Math\'ematique, Universit\'e de Rennes 1, 35042
    Rennes, France} \email{martin.costabel@univ-rennes1.fr}
\author{Fr\'ed\'erique Le Lou\"er}
 \address{ Institut f\"ur Numerische und Andgewandte Mathematik, Universit\"at G\"ottingen, 37083
   G\"ottingen, Germany} \email{f.lelouer@math.uni-goettingen.de }

\begin{abstract} We develop the shape  derivative analysis of solutions to the problem of scattering of time-harmonic electromagnetic waves by a penetrable bounded obstacle. Since boundary integral equations are a classical tool to solve  electromagnetic scattering problems, we study the shape differentiability properties of the standard electromagnetic boundary integral operators.  The latter  are typically bounded on the space of tangential vector fields of mixed regularity  $\TT\HH\sp{-\frac{1}{2}}(\Div_{\Gamma},\Gamma)$. Using Helmholtz decomposition, we can base their analysis on the study of  pseudo-differential integral operators in standard Sobolev spaces, but we then have to study the G\^ateaux differentiability of surface differential operators. We prove that the electromagnetic boundary integral operators are infinitely differentiable without loss of regularity. We  also give a characterization of the first  shape derivative  of the solution of the dielectric scattering problem as a solution of a new electromagnetic scattering problem. 
\end{abstract}      


\date{}
\maketitle

\section{Introduction}
Consider the scattering of time-harmonic electromagnetic waves by a bounded obstacle $\Omega$ in $\R^3$ with a smooth and simply connected boundary $\Gamma$ filled with an homogeneous dielectric material. This problem  is described by the system of Maxwell's equations with piecewise constant electric permittivity and magnetic permeability, valid in the sense of distributions, which implies two transmission conditions on the boundary  of the obstacle guaranteeing the continuity of the tangential components of the electric and magnetic fields across the interface.  The transmission problem is completed by the Silver--M\"uller radiation condition at infinity (see \cite{Monk} and \cite{Nedelec}).  Boundary integral equations are an efficient  method to solve such problems for low and high frequencies. The dielectric scattering  problem is usually reduced to a system of two boundary integral equations for two unknown tangential vector fields on the interface (see \cite{BuffaHiptmairPetersdorffSchwab} and \cite{Nedelec}). We refer to  \cite{CostabelLeLouer2} for methods developed by the authors to solve this problem using a single boundary integral equation.

 Optimal shape design with a goal function involving the modulus of the far field pattern of the  dielectric scattering problem has important applications, such as antenna design for telecommunication systems and radars. The analysis of shape optimization methods is based on the analysis of the dependency of the solution on the shape of the dielectric scatterer, and a local analysis involves the study of derivatives with respect to the shape. An explicit form of the shape derivatives is desirable in view of their implementation in shape optimization algorithms such as gradient methods or Newton's method.

 In this paper, we present a complete analysis of the shape differentiability of the solution of the dielectric scattering problem and of its far field pattern, using integral representations. Even if numerous works exist on the calculus of shape derivatives of various shape functionals \cite{DelfourZolesio, DelfourZolesio2,Hadamard, PierreHenrot,  Zolesio}, in the framework of boundary integral equations the scientific literature is not extensive. However, one can cite the papers  \cite{Potthast2}, \cite{Potthast1} and  \cite{Potthast3}, where  R.~Potthast has considered the question, starting with his PhD thesis \cite{Potthast4}, for the Helmholtz equation with Dirichlet or Neumann boundary conditions  and the perfect conductor problem, in spaces of continuous and H\"older continuous functions. Using the  integral representation of the solution, one is lead to study  the G\^ateaux differentiability of  boundary integral operators and potential operators with weakly singular and hypersingular  kernels.
 
  The natural space of distributions (energy space) which occurs in the electromagnetic potential theory is $\TT\HH\sp{-\frac{1}{2}}(\Div_{\Gamma},\Gamma)$, the set of tangential vector fields whose components are in the Sobolev space $H\sp{-\frac{1}{2}}(\Gamma)$ and whose surface divergence is in $H\sp{-\frac{1}{2}}(\Gamma)$. 
We face two main difficulties: 
On one hand, the solution of the scattering problem is given in terms of products of boundary integral operators and their inverses. In order 
to be able to construct shape derivatives of such products, it is not sufficient to find shape derivatives of the boundary integral operators, but it is imperative to prove that the derivatives are bounded operators between the same spaces as the boundary integral operators themselves. 
On the other hand, the very definition of shape differentiability of operators defined on the shape-dependent space $\TT\HH\sp{-\frac{1}{2}}(\Div_{\Gamma},\Gamma)$ poses non-trivial problems. Our strategy consists in using the Helmholtz decomposition of this Hilbert space which gives a representation of a tangential vector field in terms of (tangential derivatives of) two scalar potentials. In this way, we split the analysis into two steps:  First the G\^ateaux differentiability analysis of scalar boundary integral operators and potential operators with strongly and weakly singular kernels,  and second the study of shape derivatives of surface differential operators.
  
 This work contains results from the thesis \cite{FLL} where this analysis has been used to develop a shape optimization algorithm of dielectric lenses in order to obtain  a prescribed radiation pattern. 
 
This is the second of two papers on shape derivatives of boundary integral operators, the first one \cite{CostabelLeLouer} being aimed at a general theory of shape derivatives of singular integral operators appearing in boundary integral equation methods.
 
 \medskip
 
The paper is organized as follows: 

In Section \ref{BoundIntOp}  we  recall some standard results about trace mappings and regularity properties of the boundary integral operators in electromagnetism. In Section \ref{ScatProb} we define the scattering problem for time-harmonic electromagnetic waves at a dielectric interface.  We then give an integral representation of the solution --- and of the quantity of interest, namely the far field of the dielectric scattering problem --- following the single source integral equation method developed in \cite{CostabelLeLouer2}. 

The remaining parts of the paper are dedicated to the shape differentiability analysis of the solution of the dielectric scattering problem. We use the results of our first paper \cite{CostabelLeLouer} on the G\^ateaux differentiability of boundary integral operators with pseudo-homogeneous kernels. We refer to this paper for a discussion of the notion of G\^ateaux derivatives in Fr\'echet spaces and of some of their basic properties. 
In Section \ref{Helmholtzdec} we discuss the difficulties posed by the shape dependency of the function space $\TT\HH\sp{-\frac{1}{2}}(\Div_{\Gamma},\Gamma)$ on which the integral operators are defined, and we present a strategy for dealing with this difficulty, namely using the well-known tool \cite{delaBourdonnaye} of Helmholtz decomposition. 
In our approach, we map the variable spaces $\TT\HH\sp{-\frac{1}{2}}(\Div_{\Gamma_{r}},\Gamma_{r})$ to a fixed reference space with a transformation that preserves the Hodge structure. This technique involves the analysis of surface differential operators that have to be considered in suitable Sobolev spaces. Therefore in Section \eqref{SurfDiffOp} we recall and extend the results on the differentiability properties of surface differential operators established in \cite[Section 5]{CostabelLeLouer}. Using the rules on derivatives of composite and inverse functions, we obtain in Section \ref{ShapeSol} the shape differentiability properties of the solution of the scattering problem.  More precisely, we prove that  the boundary integral operators are infinitely G\^ateaux differentiable without loss of regularity, whereas previous results allowed such a loss \cite{Potthast3}, and we prove that the shape derivatives of the potentials are smooth away from the boundary but they lose regularity in the neighborhood of the boundary.  This implies that the far field is infinitely G\^ateaux differentiable, whereas the shape derivatives of the solution of the scattering problem lose regularity.

These new results generalize existing results: In the acoustic case, using a variational formulation, a characterization of the first G\^ateaux derivative was given by A. Kirsch \cite{Kirsch} for the Dirichlet problem and then  by  Hettlich  \cite{Hettlich, HettlichErra} for the impedance problem and the transmission problem. An alternative technique was introduced by Kress and P\"aiv\"arinta in \cite{KressPaivarinta} to investigate Fr\'echet differentiability   in acoustic scattering by the use of a  factorization of the difference of  the far-field pattern of the scattered wave for  two different obstacles.
In the electromagnetic case,  Potthast used the integral equation method to obtain a characterization of the first shape derivative of the solution of the perfect conductor scattering problem. In \cite{Kress}, Kress improved this result by using  a far-field identity and in \cite{HadarKress} Kress and Haddar  extended this technique to acoustic and electromagnetic impedance boundary value problems.
 
At the end of Section \ref{ShapeSol} we obtain a characterization of the first shape derivative of the solution of the dielectric scattering problem as the solution of a new electromagnetic transmission problem. We show  by deriving the integral representation of the solution that the first derivative satisfies the homogeneous Maxwell equations, and by directly deriving the boundary values of the solution itself we see that the first derivative satisfies two new  transmission conditions on the boundary. 

In the end we will have obtained two different algorithms for computing the shape derivative of the solution of the dielectric scattering problem and of the far field pattern: A first one by differentiating the integral representations and a second one by solving the new transmission problem associated with the first derivative. 

The characterization of the derivatives as solutions to boundary value problems has been obtained in the acoustic case by  Kress \cite{ColtonKress}, Kirsch \cite{Kirsch}, Hettlich and Rundell \cite{HettlichRundell} and Hohage \cite{Hohage} and has been used for the construction of  Newton-type or second degree iterative methods in acoustic inverse obstacle scattering. Whereas the use of these characterizations requires high order regularity assumption for the boundary, we expect that the differentiation of the boundary integral operators does not require much regularity.  Although in this paper we treat the case of a smooth boundary, in the last section we give some ideas on possible extensions of the results of this paper to non-smooth domains. 

\Rd
A final remark on the terminology of derivatives with respect to the variation of the domain: For the solutions of boundary value problems with a moving boundary, in the context of continuum mechanics one distinguishes frequently between \emph{material derivatives} and \emph{shape derivatives}. The former correspond to the G\^ateaux derivative with respect to the deformation, which is then interpreted as a flow associated with a velocity field, when the solution is pulled back to a fixed undeformed reference domain (Lagrangian coordinates). Depending on the interpretation of the flow, other names are used, such as ``Lagrangian derivative'', or ``substantial derivative''. The shape derivatives, on the other hand, correspond to the G\^ateaux derivative of the solution in the deformed domain (Eulerian coordinates). One also talks about ``Eulerian derivative''. The difference between the two has the form of a convection term, which can lead to a loss of one order of regularity of the shape derivative on the support of the velocity field \cite{LeugeringetalAMOptim11}. In the context of this terminology, the derivatives of the boundary integral operators that we studied in Part I of this paper \cite{CostabelLeLouer} for general pseudohomogeneous kernels and in this paper for the boundary integral operators of electromagnetism, correspond to the Lagrangian point of view, because they are obtained by pull-back to a fixed reference boundary. However, the derivative of the solution of dielectric scattering problem that we construct in Section \ref{ShapeSol} -- with the help of the derivatives of the boundary integral operators -- is the Eulerian shape derivative. We also observe the loss of one order of regularity of the shape derivative with respect to the solution of the transmission problem.
Finally, for the far field, the two notions of derivative coincide, because the deformation has compact support. 
\Bk

\section{Boundary integral operators and their main properties} \label{BoundIntOp}
Let $\Omega$ be a bounded domain in $\R^{3}$  and let $\Omega^c$ denote the exterior domain $\R^3\backslash\overline{\Omega}$. Throughout this paper, we will for simplicity assume that the boundary $\Gamma$ of $\Omega$ is a smooth and simply connected closed surface, so that $\Omega$ is diffeomorphic to a ball. 

We use standard notation for surface differential operators and boundary traces.  More details can be found in \cite{Nedelec}.
For a vector function $\vv\in\mathscr{C}^k(\R^3,\C^3)$ with $k\in\N^*$, we denote by $[\nabla\vv]$ the matrix the $i$-th column of which is the gradient of the $i$-th component of $\vv$, and we set $[\D\vv]=\transposee{[\nabla\vv]}$.
Let $\nn$ denote the outer unit normal vector on the boundary $\Gamma$.
The tangential gradient of a complex-valued scalar function $u\in\mathscr{C}^k(\Gamma,\C)$ is defined by 
\begin{equation}\label{G}
\nabla_{\Gamma}u=\nabla\tilde{u}_{|\Gamma}-\left(\nabla\tilde{u}_{|\Gamma}\cdot\nn\right)\nn,
\end{equation}
and the tangential vector curl is defined by 
\begin{equation}\label{RR}
\Rot_{\Gamma}u=\nabla\tilde{u}_{|\Gamma}\wedge\nn,
\end{equation}
where $\tilde{u}$ is a smooth extension of $u$ to the whole space $\R^3$.
For a complex-valued vector function $\uu\in\mathscr{C}^k(\Gamma,\C^3)$, we denote $[\nabla_{\Gamma}\uu]$ the matrix the $i$-th column of which is the tangential gradient of the $i$-th component of $\uu$ and we set  $[\D_{\Gamma}\uu]=\transposee{[\nabla_{\Gamma}\uu]}$.

The surface divergence of $\uu\in\mathscr{C}^k(\Gamma,\C^3)$ is defined   by 
\begin{equation}\label{e:D}
\Div_{\Gamma}\uu=\Div\tilde{\uu}_{|\Gamma}-\left([\nabla\tilde{\uu}_{|\Gamma}]\nn\cdot\nn\right),\end{equation}
 and the surface scalar curl is defined  by 
\begin{equation*}\label{R}
\rot_{\Gamma}\uu=\nn\cdot\left(\Rot\tilde{\uu}\right)\,.
\end{equation*}
These definitions do not depend on the choice of the extension $\tilde{\uu}$. 

\begin{definition}\label{2.1}
 For  a vector function $\vv\in \mathscr{C}^{\infty}(\overline{\Omega},\C^3)$ a scalar function $v\in\mathscr{C}^{\infty}(\overline{\Omega},\C)$ and $\kappa\in\C\setminus\{0\}$, we define the
traces :
$$\gamma v=v_{|_{\Gamma}} ,$$
$$\gamma_{n} v=\frac{\partial}{\partial\nn}v=\nn\cdot(\nabla v)_{|_{\Gamma}},$$
$$\gamma_{D}\vv:=\nn\wedge (\vv)_{|_{\Gamma}}\textrm{ (Dirichlet),}$$
$$\gamma_{N_{\kappa}}\vv:=\dfrac{1}{\kappa}\nn\wedge(\Rot \vv)_{|_{\Gamma}}\textrm{ (Neumann).}$$
\end{definition}
We define in the same way the exterior traces $\gamma^c$, $\gamma_{n}^c$, $\gamma_{D}^c$ and $\gamma_{N_{\kappa}}^c$. 

\medskip

For a domain $G\subset\R^3$ we denote by $H^s(G)$ the usual $L^2$-based Sobolev space of order $s\in\R$, and by $H^s_{\loc}(\overline G)$ the space of functions whose restrictions to any bounded subdomain $B$ of $G$ belong to $H^s(B)$. Spaces of vector functions will be denoted by boldface letters, thus 
$$
 \HH^s(G)=(H^s(G))^3\,.
$$ 
If $\D$ is a differential operator, we write:
\begin{eqnarray*}
\HH^s(\D,\Omega)& = &\{ u \in \HH^s(\Omega) : \D u \in \HH^s(\Omega)\}\\
\HH^s_{\loc}(\D,\overline{\Omega^c})&=& \{ u \in \HH^s_{\loc}(\overline{\Omega^c}) : \D u \in
\HH^s_{\loc}(\overline{\Omega^c}) \}.
\end{eqnarray*}
The space $\HH^s(\D, \Omega)$ is endowed with the natural graph norm. When $s=0$, this defines in particular the Hilbert spaces $\HH(\Rot,\Omega)$ and $\HH(\Rot\Rot,\Omega)$.
We introduce the Hilbert spaces $H^s(\Gamma)=\gamma\left(H^{s+\frac{1}{2}}(\Omega)\right),$ and $\TT\HH^{s}(\Gamma)=\gamma_{D}\left(\HH^{s+\frac{1}{2}}(\Omega)\right)$. For $s>0$,  the trace mappings 
 $$\gamma:H^{s+\frac{1}{2}}(\Omega)\rightarrow H^{s}(\Gamma),$$
 $$\gamma_{n}:H^{s+\frac{3}{2}}(\Omega)\rightarrow H^{s}(\Gamma),$$ 
 $$\gamma_{D}:\HH^{s+\frac{1}{2}}(\Omega)\rightarrow \TT\HH^{s}(\Gamma) $$  are then continuous. 
The dual of $H^s(\Gamma)$ and $\TT\HH^{s}(\Gamma)$ with respect to the $L^2$ (or $\LL^2$) scalar product is denoted by $H^{-s}(\Gamma)$ and $\TT\HH^{-s}(\Gamma)$, respectively.  
 
The surface differential operators defined above can be extended to  Sobolev spaces: For $s\in\R$ the tangential gradient and the tangential vector curl are obviously linear and continuous operators from $H^{s+1}(\Gamma)$ to $\TT\HH^s(\Gamma)$. The surface divergence and  the  surface scalar curl can then be defined on tangential vector fields by duality, extending duality relations valid for smooth functions  
\begin{equation}\label{dualgrad}
\int_{\Gamma}(\Div_{\Gamma}\jj)\cdot\varphi\,ds=-\int_{\Gamma}\jj\cdot\nabla_{\Gamma}\varphi\,ds
\qquad
\text{for all }\jj\in\TT\HH^{s+1}(\Gamma),\;\varphi\in H^{-s}(\Gamma),
\end{equation}
\begin{equation}\label{dualrot}
\int_{\Gamma}(\rot_{\Gamma}\jj)\cdot\varphi\,ds=\int_{\Gamma}\jj\cdot\Rot_{\Gamma}\varphi\,ds
\qquad\text{for all }\jj\in\TT\HH^{s+1}(\Gamma),\;\varphi\in H^{-s}(\Gamma).
\end{equation}
 We have the following equalities:
\begin{equation}\label{eqd2}
  \rot_{\Gamma}\nabla_{\Gamma}=0\text{ and }\Div_{\Gamma}\Rot_{\Gamma}=0
\end{equation}
\begin{equation}\label{eqd3}
 \Div_{\Gamma}(\nn\wedge \jj)=-\rot_{\Gamma}\jj\text{ and }\rot_{\Gamma}(\nn\wedge \jj)=\Div_{\Gamma}\jj\end{equation}

\begin{definition} Let $s\in\R$. We define the Hilbert space
$$
  \TT\HH^{s}(\Div_{\Gamma},\Gamma)=\left\{ \jj\in\TT\HH^{s}(\Gamma) \,;\, \Div_{\Gamma}\jj \in H^{s}(\Gamma)\right\}
$$
endowed with the norm 
$$
||\cdot||_{\TT\HH^{s}(\Div_{\Gamma},\Gamma)}=
\left(||\cdot||_{\TT\HH^{s}(\Gamma)}^{2}+||\Div_{\Gamma}\cdot||_{H^{s}(\Gamma)}^{2}\right)^{\frac12}.
$$ 
\end{definition}

\begin{lemma}\label{2.3} 
The operators $\gamma_{D}$ and $\gamma_{N}$ are linear 
continuous from $\mathscr{C}^{\infty}(\overline{\Omega},\C^3)$ to $\TT\LL^2(\Gamma)$ and they
can be extended to continuous linear operators from $\HH(\Rot,\Omega)$ and 
$\HH(\Rot,\Omega)\cap\HH(\Rot\Rot,\Omega)$, respectively, to $\TT\HH^{-\frac{1}{2}}(\Div_{\Gamma},\Gamma)$. 
\end{lemma}

For  $\uu\in \HH_{\loc}(\Rot,\overline{\Omega^c})$ and $\vv \in \HH_{\loc}(\Rot\Rot,\overline{\Omega^c}))$
we define $\gamma_{D}^c\uu $ and $\gamma_{N}^c\vv $ in the same way and the same mapping properties hold true.

Recall that we assume that the boundary $\Gamma$ is smooth and topologically trivial. For a proof of the following result, we refer to \cite{BuffaCiarlet,MartC,Nedelec}.
\begin{lemma}\label{LapBel}
 Let $s\in\R$. The Laplace--Beltrami operator defined by 
 \begin{equation}\label{eqd1}
 \Delta_{\Gamma}u=\Div_{\Gamma}\nabla_{\Gamma}u=-\rot_{\Gamma}\Rot_{\Gamma}u.
 \end{equation}
  is linear and continuous from $H^{s+2}(\Gamma)$ to $H^s(\Gamma)$. 
For $f\in H^s(\Gamma)$ and $u\in H^{s+2}(\Gamma)$, the equation 
$\Delta_{\Gamma}u=f$ has the equivalent formulation
\begin{equation}\label{ippLB}
\int_{\Gamma}\nabla_{\Gamma}u\cdot\nabla_{\Gamma}\varphi\,ds
=
-\int_{\Gamma}f\cdot\varphi\,ds,\qquad\text{ for all }\varphi\in H^{-s}(\Gamma).
\end{equation}
The operator $\Delta_{\Gamma}:H^{s+2}(\Gamma)\to H^{s}(\Gamma)$ is Fredholm of index zero, its kernel and cokernel consisting of constant functions, so that
 $\Delta_{\Gamma}:H^{s+2}(\Gamma)\slash\C\to H^{s}_{*}(\Gamma)$ is an isomorphism. Here we define the space $H^s_{*}(\Gamma)$  by
$$
f\in H^s_{*}(\Gamma)\quad\Longleftrightarrow\quad f\in H^s(\Gamma)\textrm{ and }\int_{\Gamma}f\,ds=0.
$$
For $f\in H^s_{*}(\Gamma)$ we denote the unique solution $u\in H^{s+2}(\Gamma)\slash\C$ of \eqref{ippLB} by $u=\Delta_{\Gamma}^{-1}f$.
\end{lemma}
This result is due to the injectivity of the operator $\nabla_{\Gamma}$ from $H^{s+2}(\Gamma)\slash\C$ to $\TT\HH^{s+1}(\Gamma)$, the Lax-Milgram lemma applied to \eqref{ippLB} for $s=-1$, and standard elliptic regularity theory. 
Note that $\Rot_{\Gamma}$ is also injective from $H^{s+2}(\Gamma)\slash\C$ to $\TT\HH^{s+1}(\Gamma)$, and by duality both $\Div_{\Gamma}$ and $\rot_{\Gamma}$ are surjective from $\TT\HH^{s+1}(\Gamma)$ to $H^s_{*}(\Gamma)$

Notice  that $\rot_{\Gamma}$ is defined in a natural way on all of $\HH^{s+1}(\Gamma)$ and maps to $H^s_{*}(\Gamma)$, because we have $\rot_{\Gamma}(\varphi\nn)=0$ for any scalar function $\varphi\in H^{s+1}(\Gamma)$. Thus  \eqref{dualrot} is still valid for a not necessarily tangential density $\jj\in\HH^{s+1}(\Gamma)$. An analogous property for $\Div_{\Gamma}$ defined by \eqref{e:D} is not available.
 
\bigskip

We now recall some well known results about electromagnetic potentials. Details can be found 
in \cite{BuffaCiarlet, BuffaCostabelSchwab, BuffaCostabelSheen,BuffaHiptmairPetersdorffSchwab,HsiaoWendland,Nedelec}.

Let $\kappa$ be a positive real number  and let 
$G_{a}(\kappa,|x-y|)=\dfrac{e^{i\kappa|x-y|}}{4\pi| x-y|}
$
be the fundamental solution of the Helmholtz equation $ {\Delta u + \kappa^2u =0}.
$
The single layer potential  $\Psi_{\kappa}$ is given by   
\begin{center}$\quad\Psi_{\kappa}u(x) = \displaystyle{\int_{\Gamma}G_{a}(\kappa,|x-y|)u(y) ds(y)}\qquad x
\in\R^3\backslash\Gamma$,\end{center}
and its trace by 
$$
 V_{\kappa}u(x)= \int_{\Gamma}G_{a}(\kappa,|x-y|)u(y) ds(y)\qquad x
\in\Gamma.
$$
As discussed in the first part of this paper \cite{CostabelLeLouer}, the fundamental solution is pseudo-homogeneous of class $-1$. The single layer potential $\Psi_{\kappa}u$ is continuous across the boundary $\Gamma$. As a consequence we have the following result :
\begin{lemma}
\label{3.1}  Let $s\in\R$.
The operators
$$ 
\begin{array}{ll} \Psi_{\kappa} & : 
H^{s-\frac{1}{2}}(\Gamma)\rightarrow 
  H^{s+1}(\Omega)\cap H^{s+1}_{\loc}(\overline{\Omega^{c}}) 
  \;
  \Big(H^{s-\frac{1}{2}}(\Gamma)\rightarrow H^{s+1}_{\loc}(\R^3)\text{ if }s<\frac12\Big)
  \\[1ex]
V_{\kappa}& : H^{s-\frac{1}{2}}(\Gamma)\rightarrow H^{s+\frac{1}{2}}(\Gamma) \end{array}
$$  
are continuous.
\end{lemma}

The electric potential operator $\Psi_{E_{\kappa}}$ is defined for $\jj\in
\TT\HH^{-\frac{1}{2}}(\Div_{\Gamma},\Gamma) $ by 
$$
\Psi_{E_{\kappa}}\jj :=
\kappa\,\Psi_{\kappa}\,\jj + \dfrac{1}{\kappa}\nabla\Psi_{\kappa}\Div_{\Gamma}\jj \,.
$$
In $\R^{3}\setminus\Gamma$, this can be written as 
$\Psi_{E_{\kappa}}\jj :=
\dfrac{1}{\kappa}\Rot\Rot\Psi_{\kappa}\jj$ 
because of the Helmholtz equation and
the identity $\strut\Rot\Rot = -\Delta +\nabla\Div$.

The magnetic potential operator 
$\Psi_{M_{\kappa}}$ is defined for $\mm\in \TT\HH^{-\frac{1}{2}}(\Div_{\Gamma},\Gamma) $
by 
$$
 \Psi_{M_{\kappa}}\mm := \Rot\Psi_{\kappa}\mm.
$$
 We denote the identity operator by $\Id$.
\begin{lemma}
\label{3.2}
The potentials operators $\Psi_{E_{\kappa}}$ and $\Psi_{M_{\kappa}}$ are
continuous from $\TT\HH^{-\frac{1}{2}}(\Div_{\Gamma},\Gamma)$ to
$\HH_{\loc}(\Rot,\R^3)$. 
For $\jj\in \TT\HH^{-\frac{1}{2}}(\Div_{\Gamma},\Gamma)$ we have 
$$
(\Rot\Rot -\kappa^2\Id)\Psi_{E_{\kappa}}\jj = 0\textrm{ and }(\Rot\Rot
-\kappa^2\Id)\Psi_{M_{\kappa}}\mm = 0\textrm{ in }\R^3\backslash\Gamma,
$$ and  $\Psi_{E_{\kappa}}\jj$
and $\Psi_{M_{\kappa}}\mm$ satisfy the Silver-M\"uller condition.
\end{lemma}
We define the electric and the magnetic far field operators for a density $\jj$ and an element $\hat{x}$ of the unit sphere $S^2$ of $\R^3$ by
\begin{equation}\label{FF}
\begin{split}
 \Psi_{E_{\kappa}}^{\infty}\,\jj(\hat{x})=&\;\kappa\;\hat{x}\wedge\left(\int_{\Gamma}e^{-i\kappa\hat{x}\cdot y}\jj(y)ds(y)\right)\wedge\hat{x},\\\Psi_{M_{\kappa_{e}}}^{\infty}\,\jj(\hat{x})=&\;i\kappa\;\hat{x}\wedge\left(\int_{\Gamma}e^{-i\kappa\hat{x}\cdot y}\jj(y)ds(y)\right).
\end{split}
\end{equation}
 These operators are bounded from $\TT\HH^{s}(\Div_{\Gamma},\Gamma)$ to \\
 \mbox{$\TT\LL^2(S^2)=\{\hh\in\LL^2(S^2);\;\hh(\hat{x})\cdot\hat{x}=0\}$}, for all $s\in\R$.

We can now define the main boundary integral operators:
\begin{align}
C_{\kappa}\jj(x)&=-\displaystyle{\int_{\Gamma}\nn(x)\wedge\Rot\Rot^x\{G_{a}(\kappa,|x-y|)\jj(y)\}ds(y)}\nonumber\\
\label{Ck}
 &=\left(-\kappa\;\;\nn\wedge V_{\kappa}\,\jj+\dfrac{1}{\kappa}\Rot_{\Gamma}V_{\kappa}\Div_{\Gamma}\jj\right)(x),\\
 \intertext{\quad\; and}
 M_{\kappa}\jj(x)&=-\displaystyle{\int_{\Gamma}\nn(x)\wedge\Rot^x\{G_{a}(\kappa,|x-y|)\jj(y)\}ds(y)}\nonumber\\
\label{MkDkBk}
 &=\;\;({D_{\kappa}}\,\jj-{B_{\kappa}}\,\jj)(x),\\
\intertext{\quad\; with}
{B_{\kappa}}\,\jj(x)&=\displaystyle{\int_{\Gamma}\nabla^xG_{a}(\kappa,|x-y|)\left(\jj(y)\cdot\nn(x)\right)ds(y),}\nonumber\\
{D_{\kappa}}\,\jj(x)&=\displaystyle{\int_{\Gamma}\left(\nabla^xG_{a}(\kappa,|x-y|)\cdot\nn(x)\right)\jj(y)ds(y).}\nonumber
\end{align}

The operators $M_{\kappa}$ and $C_{\kappa}$ are bounded operators from
$\TT\HH^{-\frac{1}{2}}(\Div_{\Gamma},\Gamma)$ to itself.
 
\section{The dielectric scattering problem} \label{ScatProb}

We consider the scattering of time-harmonic waves at a fixed frequency $\omega$ by a three-dimensional bounded and  non-conducting homogeneous dielectric  obstacle represented by the domain $\Omega$.  The electric permittivity $\epsilon$ and the magnetic permeability $\mu$ are assumed to take  constant positive real values in $\Omega$ and $\Omega^c$. Thus they will be discontinuous across the interface $\Gamma$, in general.  The  wave number is given by $\kappa=\omega\sqrt{\mu\epsilon}$. We  distinguish the dielectric quantities related to the interior domain $\Omega$ through the index $i$ and to the exterior domain $\Omega^c$ through the index $e$.
The time-harmonic Maxwell system can be reduced to second order equations for the electric field only.  The time-harmonic dielectric scattering problem is then formulated as follows.
\medskip
 
\textbf{The solution of the dielectric scattering problem :} 
Consider the scattering of a given incident electric wave $\EE^{inc}\in\HH_{\loc}(\Rot,\R^3)$ that satisfies 
$\Rot\Rot \EE^{inc} - \kappa_{e}^2\EE^{inc} =0$ in a neighborhood of $\overline{\Omega}$. The
interior electric   field $\EE^{i}\in \HH(\Rot,\Omega)$ and the exterior electric scattered field $\EE^{s}\in
\HH_{\loc}(\Rot,\overline{\Omega^c})$ satisfy the time-harmonic Maxwell equations
\begin{eqnarray}
\label{(1.2a)}
  \Rot\Rot \EE^{i} - \kappa_{i}^2\EE^{i}& = 0&\text{ in }\Omega,
\\
\label{(1.2b)} 
  \Rot\Rot \EE^{s} - \kappa_{e}^2\EE^{s} &= 0&\text{ in }\Omega^c,
  \end{eqnarray}
the two transmission conditions, 
\begin{eqnarray}
\label{T1}&\,\,\;\;\nn\wedge \EE^{i}=\nn\wedge( \EE^{s}+\EE^{inc})&\qquad\text{ on }\Gamma
\\
\label{T2} &\mu_{i}^{-1}(\nn\wedge\Rot \EE^{i}) = \mu_{e}^{-1}\nn\wedge\Rot(\EE^{s}+\EE^{inc})&\qquad\text{ on }\Gamma
\end{eqnarray} 
and the Silver-M\"uller radiation condition:
\begin{equation}\label{T3}
\lim_{|x|\rightarrow+\infty}|x|\left| \Rot \EE^{s}(x)\wedge\frac{x}{| x |}- i\kappa_{e}\EE^{s}(x) \right| =0.
\end{equation}
 
It is well known that the problem \eqref{(1.2a)}-\eqref{T3} admits a unique solution for any positive real values of the exterior wave number $\kappa_{e}$. We refer the reader to \cite{BuffaHiptmairPetersdorffSchwab, CostabelLeLouer2,MartinOla} for a proof via boundary integral equation methods.\medskip

To analyze the  dependency of the solution on the shape of the scatterer $\Omega$, we will use an  integral representation of the solution,  obtained by the single boundary integral equation method developped by the autors in \cite{CostabelLeLouer2}. It is based on the layer ansatz
for the exterior  electric field $\EE^s$:
\begin{equation} \label{u2c}
\EE^{s} =  -{\Psi_{E_{\kappa_{e}}}}\jj -i\eta{\Psi_{M_{\kappa_{e}}}}{C_{0}^*}\jj\quad \text{ in }\R^3\setminus\overline{\Omega} 
\end{equation}
where $\eta$ is a positive real number, 
$\jj\in \TT\HH^{-\frac{1}{2}}(\Div_{\Gamma},\Gamma)$ and  the operator ${C}_{0}^{*}$ is defined for $\jj\in \TT\HH^{-\frac{1}{2}}(\Div_{\Gamma},\Gamma)$ by 
$$
C_{0}^*\jj=-\nn\wedge V_{0}\,\jj-\Rot_{\Gamma}V_{0}\Div_{\Gamma}\jj.
$$
Thanks to the transmission conditions and the Stratton-Chu formula, we have the integral representation of the interior field
\begin{equation} \label{u12} \EE^{i} =-\frac{1}{\rho}(\Psi_{E_{\kappa_{i}}}\{\gamma_{N_{e}}^c\EE^{inc} +N_{e}\jj\}) -
(\Psi_{M_{\kappa_{i}}}\{\gamma_{D}^c\EE^{inc} + L_{e}\jj\})\text{ in }\Omega\end{equation}
 where $\rho=\dfrac{\kappa_{i}\mu_{e}}{\kappa_{e}\mu_{i}}$ and 
 $$ L_{e}={C_{\kappa_{e}}}+i\eta\left(-\frac{1}{2}\Id+{M_{\kappa_{e}}}\right){C_{0}^*},$$ $$N_{e}=\left(-\frac{1}{2}\Id+{M_{\kappa_{e}}}\right)+i\eta {C_{\kappa_{e}}}{C_{0}^*}.$$
The exterior Dirichlet trace applied to the right-hand side \eqref{u12} vanishes. The density $\jj$ then solves the following boundary integral equation

\begin{equation} \label{SS}
{\SS} \jj \equiv \rho\left(-\tfrac{1}{2}\Id + M_{\kappa_{i}}\right)L_{e}\jj + C_{\kappa_{i}}N_{e}\jj = -\rho\left(-\tfrac{1}{2}\Id + M_{\kappa_{i}}\right)\gamma_{D}\EE^{inc}+ C_{\kappa_{i}}\gamma_{N_{\kappa_{e}}}\EE^{inc}.
\end{equation} 

 \begin{theorem}
\label{thT}
  The operator $\SS$  from $\TT\HH^{-\frac{1}{2}}(\Div_{\Gamma},\Gamma)$ to itself is linear, bounded and invertible. Moreover, given the electric incident field  $\EE^{inc}\!\in\HH_{\loc}(\Rot,\R^3)$, the integral representations {\rm \eqref{u2c}, \eqref{u12}} of $\EE^{i}$ and $\EE^{s}$ give the unique solution of the dielectric scattering problem  for all positive real values of the dielectric constants $\mu_{i}$, $\mu_{e}$, $\epsilon_{i}$ and $\epsilon_{e}$.
 \end{theorem}

An important quantity that is of interest in many shape optimization problems is the far field pattern of the electric field, defined on the unit sphere of $\R^3$ by
$$
\EE^{\infty}(\hat{x})=\lim_{|x|\rightarrow\infty}4\pi|x|\dfrac{\EE^s(x)}{e^{i\kappa_{e}|x|}},\qquad \text{ with }\dfrac{x}{|x|}=\hat{x}.
$$
We have $\EE^{\infty}\in\TT\LL^2(S^2)\cap\mathscr{C}^{\infty}(S^2,\C^3)$.
To obtain the integral representation of the far field $\EE^{\infty}$ of the solution, it suffices to replace in \eqref{u2c} the potential operators $\Psi_{E_{\kappa_{e}}}$ and $\Psi_{M_{\kappa_{e}}}$ by the far field operators $\Psi_{E_{\kappa_{e}}}^{\infty}$ and $\Psi_{M_{\kappa_{e}}}^{\infty}$ defined in  \eqref{FF}, respectively.

In the method we have described, the solution $\EE=(\EE^{i},\EE^{s})$ and the far field $\EE^{\infty}$ are constructed from operators defined by integrals on the boundary $\Gamma$ and the incident field. For a fixed incident field and fixed constants $\kappa_{i}$, $\kappa_{e}$, $\mu_{i}$, $\mu_{e}$,  these quantities therefore depend on the geometry of the boundary $\Gamma$ of the scatterer $\Omega$ only.  In the sequel we analyze the $\Gamma$-dependence of the solution following the definition of  shape derivatives and the notations of section 4 of the paper \cite{CostabelLeLouer}.

\section{Shape dependence via Helmholtz decomposition}\label{Helmholtzdec}

Let us fix a reference domain $\Omega$. We consider  variations of $\Omega$ generated by transformations of the form $x\mapsto x+r(x)$ of points $x\in\R^3$, where $r$ is a smooth vector function defined on $\Gamma$. This transformation deforms the domain $\Omega$  in a  domain $\Omega_{r}$ of boundary $\Gamma_{r}$. The functions $r$ are assumed to belong to the Fr\'echet space $\mathscr{C}^{\infty}(\Gamma,\R^3)$. For $\varepsilon>0$ and some metric $d_{\infty}$ on $\mathscr{C}^{\infty}(\Gamma,\R^3)$, we set
$$
 B^{\infty}(0,\varepsilon)= \left\{r\in\mathscr{C}^{\infty}(\Gamma,\R^3),\; d_{\infty}(0,r)<\varepsilon\right\}.
$$ 
In the following, we choose $\varepsilon$ small enough so that for any $r\in B^{\infty}(0,\varepsilon)$, $(\Id+r)$ is a diffeomorphism from $\Gamma$ to  $\Gamma_{r}=(\Id+r)\Gamma=\left\{x_{r}=x+r(x); x\in\Gamma\right\}$.

The aim of this paper is to study the shape differentiability, that is, the G\^ateaux differentiability with respect to $r$, of the functionals mapping $r$ to the solution 
$$
  (\mathscr{E}^{i}(r),\mathscr{E}^{s}(r))=\left(\EE^i(\Gamma_{r}),\EE^s(\Gamma_{r})\right)
$$
of the dielectric scattering problem with obstacle $\Omega_{r}$, and to the far field 
$\mathscr{E}^{\infty}(r)=\EE^\infty(\Gamma_{r})$.

In the following, we use the superscript $r$ for integral operators and trace mappings pertaining to $\Gamma_{r}$, while functions defined on $\Gamma_{r}$ will often have a subscript $r$. According to the boundary integral equation method described in the previous section, we have 
\begin{equation}
\mathscr{E}^s(r)= \left(-\Psi_{E_{\kappa_{e}}}^r-i\eta\Psi_{M_{\kappa_{e}}}^rC_{0}^{*,r}\right)\jj_{r}\qquad\text{ in }\Omega_{r}^c=\R^3\backslash\overline{\Omega_{r}},
\end{equation}
where $\jj_{r}$ solves the integral equation
$$
\SS^r\jj_{r}=-\rho\left(-\frac{1}{2}\Id+M_{\kappa_{i}}^r\right)\gamma_{D}^r\EE^{inc}-C_{\kappa_{i}}^r\gamma_{N_{\kappa_{e}}}^r\EE^{inc},
$$
and 
\begin{equation}\label{risr2} 
 \mathscr{E}^{i}(r) = -\frac{1}{\rho}\Psi_{E_{\kappa_{i}}}^r\gamma_{N_{\kappa_{e}}}^{c,r}\mathscr{E}^{tot}(r) - \Psi_{M_{\kappa_{i}}}^r\gamma_{D}^{c,r}\mathscr{E}^{tot}(r)\qquad\text{ in }\Omega_{r},
\end{equation}
with 
\begin{equation}\label{risr1}
\mathscr{E}^{tot}(r)=\EE^{inc}+\mathscr{E}^s(r)\,.
\end{equation}
The far field pattern of the dielectric scattering problem by the interface $\Gamma_{r}$ is
 $$
\mathscr{E}^{\infty}(r)= \left(-\Psi^{\infty,r}_{E_{\kappa_{e}}}-i\eta\Psi_{M_{\kappa_{e}}}^{\infty,r}C_{0}^{*,r}\right)\jj_{r}.
 $$
 As defined in \eqref{SS}, the operator  $\SS^r$  is composed of the operators  $C_{\kappa_{e}}^r$, $M_{\kappa_{e}}^r$, $C_{\kappa_{i}}^r$ et $M_{\kappa_{i}}^r$, which are all bounded operators on the space $\TT\HH^{-\frac{1}{2}}(\Div_{\Gamma_{r}},\Gamma_{r})$. Therefore we have to study  the G\^ateaux differentiability of  the following mappings on $B^{\infty}(0,\varepsilon)$
 $$
 \begin{array}{ll}
 r\mapsto M_{\kappa}^r, C_{\kappa}^r\in\mathscr{L}(\TT\HH^{-\frac{1}{2}}(\Div_{\Gamma_{r}},\Gamma_{r}))
 \\
 r\mapsto \Psi_{M_{\kappa}}^r,\Psi_{E_{\kappa}}^r\in\mathscr{L}(\TT\HH^{-\frac{1}{2}}(\Div_{\Gamma_{r}},\Gamma_{r}),\HH(\Rot,\Omega_{r})\cup\HH_{loc}(\Rot,\overline{\Omega_{r}^c}))
 \\
 r\mapsto \Psi^{\infty,r}_{M_{\kappa}},\Psi^{\infty,r}_{E_{\kappa}}\in\mathscr{L}(\TT\HH^{-\frac{1}{2}}(\Div_{\Gamma_{r}},\Gamma_{r}),\TT\LL^2(S^2)\cap\mathscr{C}^{\infty}(S^2,\C^3)).
\end{array}
$$
Finally, the differentiability properties of the mapping $r\mapsto C_{0}^{*r}$ can be deduced from those of the mapping $r\mapsto C_{\kappa}^r$.

In this approach, several difficulties have to be overcome. The first one is that if we want to find the derivatives of the solution of the scattering problem, which is given as a product of operators and of their inverses, all defined on the same space $\TT\HH^{-\frac{1}{2}}(\Div_{\Gamma},\Gamma)$ (for the derivative at $r=0$), it is necessary to prove that the derivatives themselves are defined as bounded operators on the same space, too. On the other hand, the very definition of the differentiability of operators defined on $\TT\HH\sp{-\frac{1}{2}}(\Div_{\Gamma},\Gamma)$ raises non-trivial questions. For reducing the variable space  $\TT\HH^{-\frac{1}{2}}(\Div_{\Gamma_{r}},\Gamma_{r})$ to a fixed reference space, it is not sufficient, as we did in the scalar case studied in the first part \cite{CostabelLeLouer}, to use a change of variables. Let us discuss this question: \emph{How to define the shape derivative of operators defined on the variable space  $\TT\HH^{-\frac{1}{2}}(\Div_{\Gamma_{r}},\Gamma_{r})$?} in detail.
 
We recall the notation $\tau_{r}$ for the ``pullback'' induced by the change of variables. It maps a function $u_{r}$ defined on $\Gamma_{r}$ to the function $\tau_{r}u_{r}=u_{r}\circ(\Id+r)$ defined on $\Gamma$. For $r\in B^{\infty}(0,\varepsilon)$, the transformation $\tau_{r}$ is an isomorphism from   $H^t(\Gamma_{r})$ to  $H^t(\Gamma)$. We have
$$
(\tau_{r}u_{r})(x)=u_{r}(x+r(x))\text{ and }(\tau_{r}^{-1}u)(x_{r})=u(x).
$$
The natural idea to use this for a product of operators, proposed by Potthast in \cite{Potthast2} in the acoustic case, is to insert the identity  $\tau_{r}^{-1}\tau_{r}=\Id_{\HH^{-\frac{1}{2}}(\Gamma_{r})}$ between the factors. This allows to consider integral operators on the fixed  boundary $\Gamma$ only and to would require study the differentiability of the mappings
\begin{equation}
\label{taurM}
r\mapsto\tau_{r}C_{\kappa}^{r}\tau_{r}^{-1},
\quad
r\mapsto\tau_{r}M_{\kappa}^{r}\tau_{r}^{-1},
\quad
r\mapsto\Psi_{E_{\kappa}}^{r}\tau_{r}^{-1},
\quad
r\mapsto\Psi_{M_{\kappa}}^{r}\tau_{r}^{-1},
\end{equation}
but as has been already pointed out in \cite{Potthast3}, difficulties remain. The main cause for this is that $\tau_{r}$ does not map vector fields tangential to $\Gamma_{r}$ to vector fields tangential to $\Gamma$, and in particular,
$$
\tau_{r}(\TT\HH^{-\frac{1}{2}}(\Div_{\Gamma_{r}},\Gamma_{r}))\not=\TT\HH^{-\frac{1}{2}}(\Div_{\Gamma},\Gamma).
$$
This will lead to a loss of regularity if we simply try to differentiate the mappings in \eqref{taurM}. Let us explain this for the operator $M_{\kappa}$. The operator $M_{\kappa}^r$, when acting on vector fields tangential to $\Gamma_{r}$, has additional regularity like what is known for the scalar double layer potential, namely it has a pseudo-homogeneous kernel of class $-1$, whereas it is of class $0$ when considered on all vector fields (see \eqref{MkDkBk}).   
If we differentiate the kernel  of $\tau_{r}M_{\kappa}^r\tau_{r}^{-1}$, we will not obtain a pseudo-homogeneous kernel of class $-1$ on the set of vector fields tangential to $\Gamma$, so that we find a loss of regularity for the  G\^ateaux derivative of $\tau_{r}M_{\kappa}^r\tau_{r}^{-1}$.

For mapping tangent vector fields to tangent vector fields, the idea of Potthast was to use projectors from one tangent plane to the other. Let us denote by $\pi(r)$ the pullback $\tau_{r}$ followed by orthogonal projection to the tangent plane to $\Gamma$. This maps any vector function on $\Gamma_{r}$ to a tangential vector function on $\Gamma$, and we have 
$$
(\pi(r)\uu_{r})(x)=\uu_{r}(x+r(x))-\left(\nn(x)\cdot \uu_{r}(x+r(x))\right)\nn(x).
$$ 
The restriction of $\pi(r)$ to tangential functions on $\Gamma_{r}$ admits an inverse, denoted by $\pi^{-1}(r)$, if $r$ is sufficiently small. The mapping $\pi^{-1}(r)$ is defined by 
$$
(\pi^{-1}(r)\uu)(x+r(x))=\uu(x)-\nn(x)\frac{\nn_{r}(x+r(x))\cdot \uu(x)}{\nn_{r}(x+r(x))\cdot \nn(x)},
$$ 
and it is easy to see that $\pi(r)$ is an isomorphism between he space of continuous tangential vector functions on $\Gamma_{r}$ and on $\Gamma$, and for any $t$ between $\TT\HH^t(\Gamma_{r})$ and $\TT\HH^t(\Gamma)$.

In the framework of continuous tangential functions it  suffices to insert the product $\pi^{-1}(r)\pi(r)=\Id_{\TT\mathscr{C}^0(\Gamma_{r})}$ between factors in the integral representation of the solution to reduce the analysis to the study of boundary integral operators defined on $\TT\mathscr{C}^0(\Gamma)$, which does not  depend  on $r$. In our case,  we would obtain operators defined on the space
$$
\pi(r)\left(\TT\HH^{-\frac{1}{2}}(\Div_{\Gamma_{r}},\Gamma_{r})\right)=\left\{u\in \TT\HH^{-\frac{1}{2}}(\Gamma), \Div_{\Gamma_{r}}(\pi^{-1}(r)\uu)\in H^{-\frac{1}{2}}(\Gamma_{r})\right\},
$$
which still depends on $r$ and is, in general, different from 
$\TT\HH^{-\frac{1}{2}}(\Div_{\Gamma},\Gamma)$.

We propose a different approach, using the Helmholtz decomposition of the space $\TT\HH^{-\frac{1}{2}}(\Div_{\Gamma_{r}},\Gamma_{r})$ to  introduce a new   pullback operator $\Pp_{r}$ that defines an isomorphism between $\TT\HH^{-\frac{1}{2}}(\Div_{\Gamma_{r}},\Gamma_{r})$ and $\TT\HH^{-\frac{1}{2}}(\Div_{\Gamma},\Gamma)$.
 
Recall that we assume that the boundary $\Gamma$ is smooth and simply connected. We have the following decomposition. We refer to \cite{delaBourdonnaye} for the proof.
\begin{theorem} 
The Hilbert space $\TT\HH^{-\frac{1}{2}}(\Div_{\Gamma},\Gamma)$ admits the following Helmholtz decomposition:\begin{equation}
\TT\HH^{-\frac{1}{2}}(\Div_{\Gamma},\Gamma)= \nabla_{\Gamma} H^{\frac{3}{2}}(\Gamma) \oplus {\Rot}_{\Gamma} H^{\frac{1}{2}}(\Gamma) .
\end{equation}
\end{theorem} 
Since $\varepsilon$ is chosen such that for all $r\in B^{\infty}(0,\varepsilon)$ the surfaces  $\Gamma_{r}$ are still regular and simply connected, the spaces $\TT\HH^{-\frac{1}{2}}(\Div_{\Gamma_{r}},\Gamma_{r})$ admit similar decompositions.

The operator of change of variables $\tau_{r}$ is an isomorphism from $H^{\frac{3}{2}}(\Gamma_{r})$ to $H^{\frac{3}{2}}(\Gamma)$ and from $H^{\frac{1}{2}}(\Gamma_{r})$ to $H^{\frac{1}{2}}(\Gamma)$, and it maps constant functions to constant functions.
Let $\jj_{r}\in \TT\HH^{-\frac{1}{2}}(\Div_{\Gamma_{r}},\Gamma_{r})$ and let $\jj_{r}=\nabla_{\Gamma_{r}}\;p_{r}+\Rot_{\Gamma_{r}}\;q_{r}$ be its Helmholtz decomposition. The scalar functions $p_{r}$ and $q_{r}$ are determined uniquely up to additive constants. The following operator :   
\begin{equation}
\label{Prsmooth}
\begin{array}{llcl}
\Pp_{r}:&\TT\HH^{-\frac{1}{2}}(\Div_{\Gamma_{r}},\Gamma_{r})&\longrightarrow& \TT\HH^{-\frac{1}{2}}(\Div_{\Gamma},\Gamma)
\\
&\jj_{r}=\nabla_{\Gamma_{r}}\;p_{r}+\Rot_{\Gamma_{r}}\;q_{r}&\mapsto&\jj=\nabla_{\Gamma}\;(\tau_{r}p_{r})+\Rot_{\Gamma}\;(\tau_{r}q_{r})
\end{array}
\end{equation}
is therefore well defined, linear, continuous and invertible. Its inverse $\Pp_{r}^{-1}$ is given by
\begin{equation} 
\begin{array}{llcl}
\Pp_{r}^{-1}:&\TT\HH^{-\frac{1}{2}}(\Div_{\Gamma},\Gamma)&\longrightarrow &\TT\HH^{-\frac{1}{2}}(\Div_{\Gamma_{r}},\Gamma_{r})
\\
&\jj=\nabla_{\Gamma}\;p+\Rot_{\Gamma}\;q&\mapsto&\jj_{r}= \nabla_{\Gamma_{r}}\;\tau_{r}^{-1}(p)+\Rot_{\Gamma_{r}}\;\tau_{r}^{-1}(q).
\end{array}
\end{equation}
Obviously for $r=0$ we have $\Pp_{r}=\Pp_{r}^{-1}=\Id_{\TT\HH^{-\frac{1}{2}}(\Div_{\Gamma},\Gamma)}$. 
We can now insert the  identity $\Id_{\TT\HH^{-\frac{1}{2}}(\Div_{\Gamma_{r}},\Gamma_{r})}=\Pp_{r}^{-1}\Pp_{r}$ between factors in the integral  representation  of the  solution  $(\mathscr{E}^{i}(r),\mathscr{E}^s(r))$, and we are finally led to study the G\^ateaux differentiability properties of the following mappings, defined on $r$-independent spaces. 
 \begin{equation}\label{reformulation}
 \begin{array}{lclll}
 \;  B^{\infty}(0,\varepsilon) &\rightarrow& \mathscr{L}(\TT\HH^{-\frac{1}{2}}(\Div_{\Gamma},\Gamma), \HH(\Rot, K_{p})) & : & r\mapsto \Psi^{r}_{E_{\kappa}}\Pp_{r}^{-1} 
 \\ 
 \; B^{\infty}(0,\varepsilon)& \rightarrow& \mathscr{L}(\TT\HH^{-\frac{1}{2}}(\Div_{\Gamma},\Gamma), \HH(\Rot,K_{p})) & : & r\mapsto \Psi^{r}_{M_{\kappa}}\Pp_{r}^{-1}
 \\
 \;   B^{\infty}(0,\varepsilon_{p})& \rightarrow& \mathscr{L}(\TT\HH^{s}(\Div_{\Gamma},\Gamma), \TT\HH^{-\frac{1}{2}}(\Div_{\Gamma},\Gamma))& : & r\mapsto \Pp_{r}M^{r}_{\kappa}\Pp_{r}^{-1}
 \\
  \;B^{\infty}(0,\varepsilon_{p})&\rightarrow& \mathscr{L}(\TT\HH^{-\frac{1}{2}}(\Div_{\Gamma},\Gamma), \TT\HH^{-\frac{1}{2}}(\Div_{\Gamma},\Gamma))& : & r\mapsto \Pp_{r}C^{r}_{\kappa}\Pp_{r}^{-1}
\end{array}
\end{equation}
where $K_{p}$ is a compact subset of $\R^3\backslash\Gamma$.
These mappings are composed of scalar singular integral operators, the shape derivatives of which we studied in the first part \cite{CostabelLeLouer}, of surface differential operators, and of the inverse of the Laplace--Beltrami operator, which appears in the construction of the Helmholtz decomposition. 

Let us look at the representation of the operators in \eqref{reformulation} in terms of the Helmholtz decomposition.

\noindent\textbf{Helmholtz representation of $\Psi^{r}_{E_{\kappa}}\Pp_{r}^{-1}$}\\
The operator $\Psi^{r}_{E_{\kappa}}\mathbf{P_{r}}^{-1}$  is defined for $\jj=\nabla_{\Gamma}\;p+\mathbf{\Rot}_{\Gamma}\;q\in \TT\HH^{-\frac{1}{2}}(\Div_{\Gamma},\Gamma)$ and $x\in K_{p}$ by:  
\begin{equation*}\label{PE}
\begin{split}
\Psi^{r}_{E_{\kappa}}\Pp_{r}^{-1}\jj(x) =&\;\kappa\displaystyle{\int_{\Gamma_{r}}G_{a}(\kappa,|x-y_{r}|)\left(\nabla_{\Gamma_{r}}\tau_{r}^{-1}p\right)(y_{r})ds(y_{r})}
\\
&\;+\kappa\displaystyle{\int_{\Gamma_{r}}G_{a}(\kappa,|x-y_{r}|)\left(\Rot_{\Gamma_{r}} \tau_{r}^{-1}q\right)(y_{r})ds(y_{r})} 
\\
 &\;+ \dfrac{1}{\kappa}\nabla\displaystyle{\int_{\Gamma_{r}}G_{a}(\kappa,|x-y_{r}|)\left(\Delta_{\Gamma_{r}}\tau_{r}^{-1} p\right)(y_{r})ds(y_{r})}.
 \end{split}
 \end{equation*}
\medskip

\noindent\textbf{Helmholtz representation of  $\Psi^{r}_{M_{\kappa}}\Pp_{r}^{-1}$}\\
 The operator $\Psi^{r}_{M_{\kappa}}\mathbf{P_{r}}^{-1}$  is defined for $\jj=\nabla_{\Gamma}\;p+\mathbf{\Rot}_{\Gamma}\;q\in \TT\HH^{-\frac{1}{2}}(\Div_{\Gamma},\Gamma)$ and $x\in K_{p}$ by:   
\begin{equation*}\label{PM}
\begin{split}
\Psi^{r}_{M_{\kappa}}\Pp_{r}^{-1}\jj(x)=&\;\Rot\displaystyle{\int_{\Gamma_{r}}G_{a}(\kappa,|x-y_{r}|)\left(\nabla_{\Gamma_{r}}\tau_{r}^{-1}p\right)(y_{r})ds(y_{r})}
\\
&+\Rot\displaystyle{\int_{\Gamma_{r}}G_{a}(\kappa,|x-y_{r}|)\left(\mathbf{\Rot}_{\Gamma_{r}}\tau_{r}^{-1} q\right)(y_{r})(y_{r})ds(y_{r})}.
\end{split}
\end{equation*}
\medskip

\noindent\textbf{Helmholtz representation of $\Pp_{r}C^{r}_{\kappa}\Pp_{r}^{-1}$}\\ 
Recall that for $\jj_{r}\in\TT\HH^{-\frac{1}{2}}(\Div_{\Gamma_{r}},\Gamma_{r})$, the operator $C_{\kappa}^r$ is defined by
\begin{equation*}
\begin{array}{ll}
C^{r}_{\kappa}\jj_{r}(x_{r})=&-\kappa\,\nn_{r}(x_{r})\wedge\displaystyle{\int_{\Gamma_{r}}G_{a}(\kappa,|x_{r}-y_{r}|)\jj_{r}(y_{r})ds(y_{r})}
\vspace{2mm}\\
&-\dfrac{1}{\kappa}\nn_{r}(x_{r})\wedge\nabla_{\Gamma_{r}}^{x_{r}}\displaystyle{\int_{\Gamma_{r}}G_{a}(\kappa,|x_{r}-y_{r}|)\Div_{\Gamma_{r}}\jj_{r}(y_{r})ds(y_{r})}.
\end{array}
\end{equation*} 
We want to write  ${C^{r}_{\kappa}}\jj_{r}$ in the form $\nabla_{\Gamma_{r}}P_{r}+\Rot_{\Gamma_{r}}Q_{r}$. 
Using formulas \eqref{eqd2}--\eqref{eqd3}, we find
$$ 
\Div_{\Gamma_{r}}{C^{r}_{\kappa}}\jj_{r}=\Delta_{\Gamma_{r}} P_{r}\;\text{ and } \rot_{\Gamma_{r}}{C^{r}_{\kappa}}\jj_{r}=-\Delta_{\Gamma_{r}} Q_{r}.
$$
As a consequence we have for $x_{r}\in\Gamma_{r}$
\begin{equation}
\begin{array}{ll}
P_{r}(x_{r})=&-\kappa\;\Delta_{\Gamma_{r}}^{-1}\Div_{\Gamma_{r}}\left(\nn_{r}(x_{r})\wedge\displaystyle{\int_{\Gamma_{r}}G_{a}(\kappa,|x_{r}-y_{r}|)\jj_{r}(y_{r})ds(y_{r})}\right)
\end{array}
\end{equation}
and
\begin{equation*}
\begin{array}{rl}
Q_{r}(x_{r})=&-\kappa\;(-\Delta_{\Gamma_{r}}^{-1})\rot_{\Gamma_{r}}\left(\nn_{r}(x_{r})\wedge\displaystyle{\int_{\Gamma_{r}}G_{a}(\kappa,|x_{r}-y_{r}|)\jj_{r}(y_{r})ds(y_{r})}\right)
\\
&\mkern-20mu
-\dfrac{1}{\kappa}(-\Delta_{\Gamma_{r}})\rot_{\Gamma_{r}}(-\Rot_{\Gamma_{r}})\displaystyle{\int_{\Gamma_{r}}\!\!G_{a}(\kappa,|x_{r}-y_{r}|)\Div_{\Gamma_{r}}\jj_{r}(y_{r})ds(y_{r})}
\\
=&\kappa\;\Delta_{\Gamma_{r}}^{-1}\rot_{\Gamma_{r}}\left(\nn_{r}(x_{r})\wedge\displaystyle{\int_{\Gamma_{r}}G_{a}(\kappa,|x_{r}-y_{r}|)\jj_{r}(y_{r})ds(y_{r})}\right)
\\
&+\dfrac{1}{\kappa}\displaystyle{\int_{\Gamma_{r}}G_{a}(\kappa,|x_{r}-y_{r}|)\Div_{\Gamma_{r}}\jj_{r}(y_{r})ds(y_{r})}.
\end{array}
\end{equation*}
The operator $\Pp_{r}C^{r}_{\kappa}\Pp_{r}^{-1}$ is defined for $\jj=\nabla_{\Gamma}\;p+\mathbf{\Rot}_{\Gamma}\;q\in \TT\HH^{-\frac{1}{2}}(\Div_{\Gamma},\Gamma)$ by 
$$
\Pp_{r}C^{r}_{\kappa}\Pp_{r}^{-1}=\nabla_{\Gamma}P(r)+\Rot_{\Gamma}Q(r),
$$ 
with
\begin{multline*}
P(r)(x)=\\
 -\kappa\left(\tau_{r}\Delta_{\Gamma_{r}}^{-1}\Div_{\Gamma_{r}}\tau_{r}^{-1}\right)\Big((\tau_{r}\nn_{r})(x)\wedge\tau_{r}\Big\{\displaystyle{\int_{\Gamma_{r}}G_{a}(\kappa,|\cdot-y_{r}|)(\nabla_{\Gamma_{r}}\tau_{r}^{-1}p)(y_{r})ds(y_{r})}
\\
 +\displaystyle{\int_{\Gamma_{r}}G_{a}(\kappa,|\cdot-y_{r}|)(\Rot_{\Gamma_{r}}\tau_{r}^{-1}q)(y_{r})ds(y_{r})}\Big\}(x)\Big)
\end{multline*}
and
\begin{multline*}
Q(r)(x)=\\
\kappa\left(\tau_{r}\Delta_{\Gamma_{r}}^{-1}\rot_{\Gamma_{r}}\tau_{r}^{-1}\right)\Big((\tau_{r}\nn_{r})(x)\wedge\tau_{r}\Big\{\displaystyle{\int_{\Gamma_{r}}
\!\!
G_{a}(\kappa,|\cdot-y_{r}|)(\nabla_{\Gamma_{r}}\tau_{r}^{-1}p)(y_{r})ds(y_{r})}
\\
 +\displaystyle{\int_{\Gamma_{r}}G_{a}(\kappa,|\cdot-y_{r}|)(\Rot_{\Gamma_{r}}\tau_{r}^{-1}q)(y_{r})ds(y_{r})}\Big\}(x)\Big)
\\
+\dfrac{1}{\kappa}\tau_{r}\Big(\displaystyle{\int_{\Gamma_{r}}G_{a}(\kappa,|\cdot-y_{r}|)(\Delta_{\Gamma_{r}}\tau_{r}^{-1}p)(y_{r})ds(y_{r})}\Big)(x).
\end{multline*}
\medskip

\noindent\textbf{Helmholtz representation of $\Pp_{r}M^{r}_{\kappa}\Pp_{r}^{-1}$}\\ 
Recall that for all $\jj_{r}\in\TT\HH^{-\frac{1}{2}}(\Div_{\Gamma_{r}},
\Gamma_{r})$, the operator $M_{\kappa}^r$ is defined by
\begin{equation*}
\begin{array}{ll}
M^{r}_{\kappa}\jj_{r}(x_{r})=&\displaystyle{\int_{\Gamma_{r}}\big((\nabla^{x_{r}}G_{a}(\kappa,|x_{r}-y_{r}|))\cdot\nn_{r}(x_{r})\big)\jj_{r}(y_{r})ds(y_{r})}
\vspace{2mm}\\
&-\displaystyle{\int_{\Gamma_{r}}\nabla^{x_{r}}G_{a}(\kappa,|x_{r}-y_{r}|)\big(\nn_{r}(x_{r})\cdot\jj_{r}(y_{r})\big) ds(y_{r})}.
\end{array}
\end{equation*}
Using the equalities \eqref{eqd3} and the identity $\Rot\Rot=-\Delta+\nabla\Div$, we have
$$
\begin{array}{rl} 
\Div_{\Gamma_{r}}M^r_{\kappa}\jj_{r}(x_{r})=&\nn_{r}(x_{r})\cdot\!\displaystyle{\int_{\Gamma_{r}}
\!\!
\Rot\Rot^{x_{r}} \left\{G_{a}(\kappa,|x_{r}-y_{r}|)\jj_{r}(y_{r}) \right\} ds(y_{r})}
\\ 
=&\kappa^2\nn_{r}(x_{r})\cdot\displaystyle{\int_{\Gamma_{r}}\left\{G_{a}(\kappa,|x_{r}-y_{r}|)\jj_{r}(y_{r}) \right\}ds(y_{r})} 
\vspace{2mm}\\
 & +\displaystyle{\int_{\Gamma_{r}}\dfrac{\partial}{\partial\nn_{r}(x_{r})} \left\{G_{a}(\kappa,|x_{r}-y_{r}|)\Div_{\Gamma_{r}}\jj_{r}(y_{r}) \right\}ds(y_{r})}.
\end{array}
$$
 Proceeding in the same way as with the operator $\Pp_{r}C^{r}_{\kappa}\Pp_{r}^{-1}$, we obtain that the operator $\Pp_{r}M^{r}_{\kappa}\Pp_{r}^{-1}$ is defined for  $\jj=\nabla_{\Gamma}\;p+\mathbf{\Rot}_{\Gamma}\;q\in \TT\HH^{-\frac{1}{2}}(\Div_{\Gamma},\Gamma)$ by: 
 $$
 \Pp_{r}M^{r}_{\kappa}\Pp_{r}^{-1}\jj=\nabla_{\Gamma}P'(r)+\Rot_{\Gamma}Q'(r),
 $$ 
 with
\begin{multline*}
P'(r)(x)=\\
\left(\tau_{r}\Delta_{\Gamma_{r}}^{-1}\tau_{r}^{-1}\right)\tau_{r}\left\{\kappa^2 \int_{\Gamma_{r}}\nn_{r}\cdot \left\{G_{a}(\kappa,|\cdot-y_{r}|)\Rot_{\Gamma_{r}}\tau_{r}^{-1}q(y_{r})\right\}ds(y_{r}) \right.
\\
+\,\kappa^{2}\int_{\Gamma_{r}}\nn_{r}\cdot \left\{G_{a}(\kappa,|\cdot-y_{r}|)\nabla_{\Gamma_{r}}\tau_{r}^{-1}p(y_{r})\right\}ds(y_{r})
\\
+\left.\int_{\Gamma_{r}}\dfrac{\partial}{\partial\nn_{r}}G_{a}(\kappa,|\cdot-y_{r}|)(\Delta_{\Gamma_{r}}\tau_{r}^{-1}p)(y_{r})ds(y_{r})\right\}(x),
\end{multline*}
and
\begin{multline*}
Q_{r}'(x)=\\
\left(\tau_{r}\Delta_{\Gamma_{r}}^{-1}\rot_{\Gamma_{r}}\tau_{r}^{-1}\right)\tau_{r}\left\{\int_{\Gamma_{r}}\left(\nabla G_{a}(\kappa,|\cdot-y_{r}|)\cdot\nn_{r}\right)(\Rot_{\Gamma_{r}}\tau_{r}^{-1}q)(y_{r})ds(y_{r}) \right.
\\
+\int_{\Gamma_{r}}(\left( \nabla G_{a}(\kappa,|\cdot-y_{r}|)\cdot\nn_{r}\right)(\nabla_{\Gamma_{r}}\tau_{r}^{-1} p)(y_{r})ds(y_{r})
\\
- \int_{\Gamma_{r}}\nabla G_{a}(\kappa,|\cdot-y_{r}|)\left(\nn_{r}\cdot(\Rot_{\Gamma_{r}}\tau_{r}^{-1} q)(y_{r})\right) ds(y_{r})
\\
-\left.\int_{\Gamma_{r}}\nabla G_{a}(\kappa,|\cdot-y_{r}|)\left(\nn_{r}\cdot(\nabla_{\Gamma_{r}}\tau_{r}^{-1}p)(y_{r})\right) ds(y_{r})\right\}(x).
\end{multline*}
 These operators are composed of boundary integral operators with weakly singular kernel and of the surface differential operators defined in section \ref{BoundIntOp}. Each of these weakly singular boundary integral operators  has a pseudo-homogeneous kernel of class -1. The $\mathscr{C}^{\infty}$-G\^ateaux differentiability properties of such boundary integral operators has been established in the preceding paper \cite{CostabelLeLouer}. It remains now to show that the surface differential operators,  more precisely $\tau_{r}\nabla_{\Gamma_{r}}\tau_{r}^{-1},\;\tau_{r}\Rot_{\Gamma_{r}}\tau_{r}^{-1},\;\tau_{r}\Div_{\Gamma_{r}}\tau_{r}^{-1},\;\tau_{r}\rot_{\Gamma_{r}}\tau_{r}^{-1}$, as well as $\tau_{r}\Delta_{\Gamma_{r}}\tau_{r}^{-1}$ and its inverse, preserve their mapping properties  by differentiation with respect to $r$.

\section{G\^ateaux differentiability of surface differential operators}\label{SurfDiffOp}
The analysis of the surface differential operators requires the differentiability properties of some auxiliary functions, such as the outer unit normal vector $\nn_{r}$ and the Jacobian $J_{r}$ of the change of variable $x\mapsto x+r(x)$.  We recall some results established in the first part \cite[section 4]{CostabelLeLouer}. For the definition of G\^ateaux derivatives and the corresponding analysis, see \cite{Schwartz}.

\begin{lemma} \label{N} 
The mapping $\mathcal{N}:B^{\infty}(0,\varepsilon)\ni r\mapsto\tau_{r}\nn_{r}=\nn_{r}\circ(\Id+r)\in\mathscr{C}^{\infty}(\Gamma,\R^3)$ is $\mathscr{C}^{\infty}$-G\^ateaux-differentiable and its first derivative in the direction of $\xi\in\mathscr{C}^{\infty}(\Gamma,\R^3)$ is given by  
$$
d \mathcal{N}[r,\xi]=-\left[\tau_{r}\nabla_{\Gamma_{r}}(\tau_{r}^{-1}\xi)\right]\mathcal{N}(r).
$$ 
\end{lemma}

 \begin{lemma}\label{J} 
The mapping $\mathcal{J}$ from $r\in B^{\infty}(0,\varepsilon)$ to the surface Jacobian $J_{r}\in\mathscr{C}^{\infty}(\Gamma,\R)$  is $\mathscr{C}^{\infty}$-G\^ateaux differentiable and its first derivative  in the direction of $\xi\in\mathscr{C}^{\infty}(\Gamma,\R^3)$ is given by 
 $$
 d\mathcal{J}[r_{0},\xi]=J_{r_{0}}\cdot\big(\tau_{r_{0}}\Div_{\Gamma_{r_{0}}}(\tau^{-1}_{r_{0}}\xi)\big).
 $$ 
\end{lemma}

The differentiability properties of the tangential gradient and of the surface divergence in the framework of classical Sobolev spaces is established in \cite[section 5]{CostabelLeLouer}. 

\begin{lemma}\label{nabla}
The mapping 
$$
\begin{array}{cccc}
\mathcal{G}:&B^{\infty}(0,\varepsilon)&\rightarrow&\mathscr{L}(H^{s+1}(\Gamma),\HH^{s}(\Gamma))
\\
&r&\mapsto&\tau_{r}\nabla_{\Gamma_{r}}\tau_{r}^{-1}
\end{array}
$$ 
is $\mathscr{C}^{\infty}$-G\^ateaux differentiable and its first derivative for $\xi\in\mathscr{C}^{\infty}(\Gamma,\R^3)$ is given by
$$
d \mathcal{G}[r,\xi]u=-[\mathcal{G}(r)\xi]\mathcal{G}(r)u+\big(\mathcal{G}(r)u\cdot[\mathcal{G}(r)\xi]\mathcal{N}(r)\big)\,\mathcal{N}(r).
$$
\end{lemma}
 
\begin{lemma}\label{D}
The mapping 
$$
\begin{array}{cccc}
\mathcal{D}:&B^{\infty}(0,\varepsilon)&\rightarrow&\mathscr{L}(\HH^{s+1}(\Gamma),H^s(\Gamma))
\\
&r&\mapsto&\tau_{r}\Div_{\Gamma_{r}}\tau_{r}^{-1}
\end{array}
$$ 
is $\mathscr{C}^{\infty}$-G\^ateaux differentiable  and its first derivative for $\xi\in\mathscr{C}^{\infty}(\Gamma,\R^3)$ is given by
$$
d\mathcal{D}[r,\xi]\uu=-\Tr([\mathcal{G}(r)\xi][\mathcal{G}(r)\uu])+\left([\mathcal{G}(r)\uu]\mathcal{N}(r)\cdot[\mathcal{G}(r)\xi]\mathcal{N}(r)\right).
$$
\end{lemma}

Similar results can now be obtained for the tangential vector curl by composition of the tangential gradient with the normal vector.
\begin{lemma}\label{rot2}
The mapping 
$$
\begin{array}{cccc}
\boldsymbol{\mathcal{R}}:& B^{\infty}(0,\varepsilon)&\rightarrow&\mathscr{L}(H^{s+1}(\Gamma),\HH^s(\Gamma))
\\
&r&\mapsto&\tau_{r}\Rot_{\Gamma_{r}}\tau_{r}^{-1}
\end{array}
$$ 
is $\mathscr{C}^{\infty}$-G\^ateaux differentiable and its first derivative for $\xi\in\mathscr{C}^{\infty}(\Gamma,\R^3)$ is given by
$$
d \boldsymbol{\mathcal{R}}[r,\xi]u=\transposee{[\mathcal{G}(r)\xi]}\boldsymbol{\mathcal{R}}(r)u-\mathcal{D}(r)\xi\cdot\boldsymbol{\mathcal{R}}(r)u.
$$
\end{lemma}

\begin{proof} 
Let  $u\in H^{s+1}(\Gamma)$. By definition, we have $\boldsymbol{\mathcal{R}}(r)u=\mathcal{G}(r)u\wedge \mathcal{N}(r)$. By lemmas \ref{N} and \ref{nabla} this application is $\mathscr{C}^{\infty}$-G\^ateaux differentiable. For the derivative in the direction $\xi\in\mathscr{C}^{\infty}(\Gamma,\R^3)$ we find
$$
d\boldsymbol{\mathcal{R}}[r,\xi]u=-{[\mathcal{G}(r)\xi]}\mathcal{G}(r)u\wedge \mathcal{N}(r)-\mathcal{G}(r)u\wedge[\mathcal{G}(r)\xi]\mathcal{N}(r).
$$
For any $(3\times3)$ matrix $A$ and vectors $b$ and $c$ there holds 
$$
(Ab)\wedge c+b\wedge Ac=\Tr(A)(b\wedge c)-\transposee{A}(b\wedge c).
$$ 

We obtain the expression of the first derivative with the choice 
$A=-[\mathcal{G}(r)\xi]$, $b=\mathcal{G}(r)u$ and $c=\mathcal{N}(r)$.
\end{proof}

\begin{lemma}\label{rot1}
The mapping 
$$
\begin{array}{cccc}
\mathcal{R}:&B^{\infty}(0,\varepsilon)&\rightarrow&\mathscr{L}(\HH^{s+1}(\Gamma),H^s(\Gamma))
\\
&r&\mapsto&\tau_{r}\rot_{\Gamma_{r}}\tau_{r}^{-1}
\end{array}
$$ 
is $\mathscr{C}^{\infty}$-G\^ateaux differentiable and its first derivative for $\xi\in\mathscr{C}^{\infty}(\Gamma,\R^3)$ is given by
$$
d\mathcal{R}[r,\xi]\uu=-\sum_{i=1}^3\left(\mathcal{G}(r)\xi_{i}\cdot\boldsymbol{\mathcal{R}}(r)u_{i}\right)- \mathcal{D}(r)\xi\cdot \mathcal{R}(r)\uu
$$ where $\uu=(u_{1},u_{2},u_{3})$ and $\xi=(\xi_{1},\xi_{2},\xi_{3})$.
\end{lemma}

\begin{proof} 
Let $\uu\in\HH^{s+1}(\Gamma)$. 
Notice that we have $\rot_{\Gamma}\uu=-\Tr\big([\Rot_{\Gamma}\uu]\big)$. We can therefore write 
$$
\mathcal{R}(r)\uu=-\Tr(\boldsymbol{\mathcal{R}}(r)\uu).
$$  
The $\mathscr{C}^{\infty}$-differentiability of $R$ results from the $\mathscr{C}^{\infty}$-differentiability  of $\boldsymbol{\mathcal{R}}$. The first derivative in the direction $\xi$ is 
\begin{equation*}
\begin{split}
d\boldsymbol{\mathcal{R}}[r,\xi]\uu=&-\Tr\left(d \boldsymbol{\mathcal{R}}[r,\xi]\uu\right)
\\
=&-\Tr\left(\transposee{[\mathcal{G}(r)\xi]}[\boldsymbol{\mathcal{R}}(r)\uu]\right)-\mathcal{D}(r)\xi\cdot \Tr\left(-\boldsymbol{\mathcal{R}}(r)\uu]\right)
\\
=&-\sum_{i=1}^3\left(\mathcal{G}(r)\xi_{i}\cdot\boldsymbol{\mathcal{R}}(r)\uu_{i}\right)- \mathcal{D}(r)\xi\cdot \mathcal{R}(r)\uu.
\end{split}
\end{equation*}
\end{proof}
Higher order derivatives of the tangential vector curl operator and of the surface scalar curl operator can be obtained by applying these results recursively.

\medskip

In view of the integral representations of the  operators $\Pp_{r}C_{\kappa}^r\Pp_{r}^{-1}$ and $\Pp_{r}M_{\kappa}^r\Pp_{r}^{-1}$, we have to  study the G\^ateaux differentiability of the  mappings 
 $$
 \begin{array}{lll}
 r&\mapsto&\tau_{r}\Delta_{\Gamma_{r}}^{-1}\Div_{\Gamma_{r}}\tau_{r}^{-1}
 \\
 r&\mapsto&\tau_{r}\Delta_{\Gamma_{r}}^{-1}\rot_{\Gamma_{r}}\tau_{r}^{-1}.
 \end{array}
 $$ 
 We have seen that for $r\in B^{\infty}(0,\varepsilon)$ the operator $\rot_{\Gamma_{r}}$ is linear and continuous from $\HH^{s+1}(\Gamma_{r})$ to $H^s_{*}(\Gamma_{r})$, that the operator  $\Div_{\Gamma_{r}}$ is linear and continuous from $\TT\HH^{s+1}(\Gamma_{r})$ to  $H_{*}^s(\Gamma_{r})$ and that $\Delta_{\Gamma_{r}}^{-1}$ is defined from $H^{s}_{*}(\Gamma_{r})$ to $H^{s+2}(\Gamma_{r})\slash\C$. To use the chain rules,  it is  necessary to prove that the derivatives at $r=0$ act between  the spaces $\HH^{s+1}(\Gamma)$ and $H^s_{*}(\Gamma)$ for the scalar curl operator, between the spaces $\TT\HH^{s+1}(\Gamma)$ and $H^s_{*}(\Gamma)$  for the divergence operator and between the spaces   $H^s_{*}(\Gamma)$ and $H^{s+2}(\Gamma)/\C$ for the Laplace--Beltrami operator. An important observation is 
$$
u_{r}\in H^s_{*}(\Gamma_{r})\text{ if and only if }J_{r}\,u_{r}\circ(\Id+r)\in H^s_{*}(\Gamma).
$$
Using the duality  \eqref{dualrot} on the boundary $\Gamma_{r}$ we can write for any vector density $\jj\in\HH^{s+1}(\Gamma)$ and any scalar density $\varphi\in H^{-s}(\Gamma)$ 
\begin{equation}\label{rotdual}
\begin{aligned}
\int_{\Gamma}\tau_{r}\Big(\rot_{\Gamma_{r}}(\tau_{r}^{-1}\jj)\Big)\cdot \varphi \,J_{r}ds &= \int_{\Gamma_{r}}\rot_{\Gamma_{r}}(\tau_{r}^{-1}\jj)\cdot(\tau_{r}^{-1}\varphi)\,ds
\\
&=\int_{\Gamma_{r}}(\tau_{r}^{-1}\jj)\cdot\Rot_{\Gamma_{r}}(\tau_{r}^{-1}\varphi)\,ds
\\
&=\int_{\Gamma}\jj\cdot\tau_{r}\Big(\Rot_{\Gamma_{r}}(\tau_{r}^{-1}\varphi)\Big)\,J_{r}\,ds.
\end{aligned}
\end{equation}
Taking $\varphi\in\R$ (i.e. $\varphi$ is a constant function) then the right-hand side  vanishes. This means that 
  $J_{r}\big(\tau_{r}\rot_{\Gamma_{r}}(\tau_{r}^{-1}\jj)\big)$  is of vanishing mean value. 
  
\begin{lemma}
The mapping 
$$
\begin{array}{cccc}
\mathcal{R}^{*}: &B^{\infty}(0,\varepsilon)&\rightarrow&\mathscr{L}(\HH^{s+1}(\Gamma),\HH^{s}_{*}(\Gamma))
\\
&r&\mapsto&J_{r}\,\tau_{r}\rot_{\Gamma_{r}}\tau_{r}^{-1}
\end{array}
$$ 
is $\mathscr{C}^{\infty}$-G\^ateaux differentiable and we have in any direction $\xi=(\xi_{1},\xi_{2},\xi_{3})\in\mathscr{C}^{\infty}(\Gamma,\R^3)$
$$
\left\{\begin{array}{ccl}\dfrac{\partial \mathcal{R}^{*}}{\partial r}[r,\xi]&=&-\mathcal{J}(r)\cdot\sum\limits_{i=1}^3\left(\mathcal{G}(r)\xi_{i}\cdot\boldsymbol{\mathcal{R}}(r)\uu_{i}\right),\vspace{2mm}\\
\dfrac{\partial^m \mathcal{R}^{*}}{\partial r^m}[r,\xi]&=&0,\;\text{ for all }m\ge2.\end{array}\right.
$$
\end{lemma}

\begin{proof}  
Looking at the expression of the derivatives of the tangential gradient and of the tangential vector curl in lemmas \ref{nabla} and \ref{rot2} we prove iteratively that all the derivatives of $\mathcal{G}(r)\varphi$ and of $\boldsymbol{\mathcal{R}}(r)\varphi$ are composed of $\mathcal{G}(r)\varphi$ and $\boldsymbol{\mathcal{R}}(r)\varphi$, so that  for $\varphi\in\R$ the derivatives of the right-hand side of \eqref{rotdual} vanishes. We have for all  $m\in\N$ and $\uu\in\HH^{s+1}(\Gamma)$: 
$$
\frac{\partial^m}{\partial r^m}\left\{\int_{\Gamma}J_{r}\big(\tau_{r}\rot_{\Gamma_{r}}(\tau_{r}^{-1}\jj)\big) \,ds\right\}[r,\xi]=\int_{\Gamma}\frac{\partial^m}{\partial r^m}\left\{\mathcal{R}^{*}\right\}[r,\xi]\jj \,ds=0.
$$ 
It can also be obtained by directly deriving the expression of $\mathcal{R}^{*}\jj$ using the formulas obtained in the lemmas \ref{N} to \ref{rot1}.  
The first derivative of $\mathcal{R}^{*}$ is given by
$$
 \begin{array}{ll}
  d\left(\mathcal{R}^{*}\right)[r,\xi]\uu&=-\mathcal{J}(r)\cdot\sum\limits_{i=1}^3\left(\mathcal{G}(r)\xi_{i}\cdot\boldsymbol{\mathcal{R}}(r)\uu_{i}\right)\\
  &=-J_{r}.\tau_{r}\left(\sum\limits_{i=1}^3\nabla_{\Gamma_{r}}(\tau_{r}^{-1}\xi_{i})\cdot\Rot_{\Gamma_{r}}(\tau_{r}^{-1}\uu_{i})\right)
 \end{array}
$$
The right-hand side is of vanishing mean value since the space $\nabla_{\Gamma_{r}}H^s(\Gamma_{r})$ is orthogonal to $\Rot_{\Gamma_{r}}H^s(\Gamma_{r})$ for the  $\LL^2(\Gamma_{r})$ duality product. For the second order derivative we derive  $r\mapsto d\left(\mathcal{R}^{*}\right)[r,\xi]\jj$ in the direction $\eta=(\eta_{1},\eta_{2},\eta_{3})\in\mathscr{C}^{\infty}(\Gamma,\R^3)$ and we obtain

$$\begin{array}{lcl}d^2\left(\mathcal{R}^{*}\right)[r;\xi,\eta]\uu&=&\mathcal{J}(r)\cdot\sum\limits_{i=1}^3\left([\mathcal{G}(r)\eta]\mathcal{G}(r)\xi_{i}\cdot\boldsymbol{\mathcal{R}}(r)\uu_{i}\right)\\&&-\mathcal{J}(r)\cdot\sum\limits_{i=1}^3\left(\mathcal{G}(r)\xi_{i}\cdot\transposee{[\mathcal{G}(r)\eta]}\boldsymbol{\mathcal{R}}(r)\uu_{i}\right)\\&=&0.\end{array}$$
Higher order derivatives of $\mathcal{R}^{*}$ vanish.
\end{proof}

For the surface divergence, similar arguments can be applied. Using the duality  \eqref{dualgrad}, we can  write for $\jj\in\TT\HH^{s+1}(\Gamma)$: 
$$
\begin{aligned}
\displaystyle{\int_{\Gamma}\tau_{r}\Big(\Div_{\Gamma_{r}}(\pi(r)^{-1}\jj)\Big)\cdot \varphi \,J_{r}ds}
  &=\displaystyle{\int_{\Gamma_{r}}\Div_{\Gamma_{r}}(\pi(r)^{-1}\jj)\cdot(\tau_{r}^{-1}\varphi)\,ds}
\\
&=\displaystyle{-\int_{\Gamma_{r}}(\pi(r)^{-1}\jj)\cdot\nabla_{\Gamma_{r}}(\tau_{r}^{-1}\varphi)\,ds}
\\
&=\displaystyle{-\int_{\Gamma}\tau_{r}(\pi(r)^{-1}\jj)\cdot\tau_{r}\Big(\nabla_{\Gamma_{r}}(\tau_{r}^{-1}\varphi)\Big)\,J_{r}ds}.
\end{aligned}
$$
This shows that for constant $\varphi\in\R$ the right-hand side vanishes and therefore  $J_{r}\big(\tau_{r}\Div_{\Gamma_{r}}(\pi^{-1}(r)\jj)\big)$  is of vanishing mean value. We obtain the following result.
\begin{lemma}
The mapping
$$
\begin{array}{cccc}
\mathcal{D}^{*}:&B^{\infty}(0,\varepsilon)&\rightarrow&\mathscr{L}(\TT\HH^{s+1}(\Gamma),\HH^{s}_{*}(\Gamma))
\\
&r&\mapsto&J_{r}\,\tau_{r}\Div_{\Gamma_{r}}\pi^{-1}(r)
\end{array}
$$ 
 is $\mathscr{C}^{\infty}$-G\^ateaux differentiable.
\end{lemma}

Now it remains to analyze the inverse of the Laplace--Beltrami operator $\Delta_{\Gamma}$.
We apply the following abstract result on the G\^ateaux derivative of the inverse in a Banach algebra.
We leave its proof to the reader.
\begin{lemma}\label{d-1} 
 Let $U$ be an open subset of a Fr\'echet space $\mathcal{X}$ and let $\mathcal{Y}$ be a Banach algebra. Assume that $f : U\rightarrow \mathcal{Y}$ is G\^ateaux differentiable at $r_{0}\in U$ and that $f(r)$ is invertible in $\mathcal{Y}$ for all $r\in U$. Then $g$ is G\^ateaux differentiable  at $r_{0}$ and its first derivative in the  direction $\xi\in\mathcal{X}$ is 
\begin{equation}d f[r_{0},\xi] = -f(r_{0})^{-1} \circ d f[r_{0},\xi] \circ f(r_{0})^{-1}.
\end{equation}
 Moreover if $f$ is $\mathscr{C}^m$-G\^ateaux differentiable then $g$ is, too.
\end{lemma}

From the preceding results we deduce the $\mathscr{C}^{\infty}$-G\^ateaux differentiability of 
the mapping 
$$
\begin{array}{cccl}
\mathcal{L}^{*} :& B^{\infty}(0,\varepsilon)&\rightarrow&\mathscr{L}(H^{s+2}(\Gamma),H^{s}_{*}(\Gamma))
\\
&r&\mapsto& J_{r}\tau_{r}\Delta_{\Gamma_{r}}\tau_{r}^{-1}=-\mathcal{R}^*(r)\boldsymbol{\mathcal{R}}(r).
\end{array}
$$
Let us note that $\tau_{r}$ induces an isomorphism between the quotient spaces 
$H^s(\Gamma_{r})/\C$  and $H^s(\Gamma)/\C$.

\begin{lemma}\label{delta-1} 
The mapping 
$$
\begin{array}{ccc}
B^{\infty}(0,\varepsilon)&\rightarrow&\mathscr{L}(H^{s}_{*}(\Gamma), H^{s+2}(\Gamma)\slash\C)
\\
r&\mapsto&\big(\mathcal{L}^{*}(r)\big)^{-1}
\end{array}
$$ 
is $\mathscr{C}^{\infty}$-G\^ateaux differentiable and we have in any direction $\xi\in\mathscr{C}^{\infty}(\Gamma,\R^3)$
\begin{equation}\label{dLapBm1}
d\left\{r\mapsto\big(\mathcal{L}^{*}(r)\big)^{-1}\right\}[0,\xi]=\;-\Delta_{\Gamma}^{-1}\circ d\mathcal{L}^{*}[0,\xi]\circ\Delta_{\Gamma}^{-1}.
\end{equation}  
\end{lemma}

\begin{proof}
We have seen in section \ref{BoundIntOp} that the Laplace--Beltrami operator is invertible  from $H^{s+2}(\Gamma_{r})\slash\C$ to $H^s_{*}(\Gamma_{r})$.  As a consequence $\mathcal{L}^{*}(r)$ is invertible from $H^{s+2}(\Gamma)\slash\C$ to $H^s_{*}(\Gamma)$. We conclude by using Lemma \ref{d-1}.
 \end{proof}

Let us give another formulation of \eqref{dLapBm1}.
For any $u\in H^{s+2}(\Gamma)$ and $\varphi\in H^{-s}(\Gamma)$ we have
$$
\int_{\Gamma}\tau_{r}\Big(\Delta_{\Gamma_{r}}(\tau_{r}^{-1}u)\Big)\cdot \varphi \,J_{r}ds=-\int_{\Gamma}\Big(\mathcal{G}(r)u\cdot\mathcal{G}(r)\varphi\Big)\,J_{r}ds.
$$
It is more convenient to differentiate the right-hand side than the left hand side.  For $f\in H^s_{*}(\Gamma_{r})$, the element $\big(\mathcal{L}^{*}(r)\big)^{-1}f$ is the solution $u$ of 
$$
-\int_{\Gamma}\Big(\mathcal{G}(r)u\cdot\mathcal{G}(r)\varphi\Big)\,J_{r}ds=\int_{\Gamma}f\cdot\varphi\,ds,\qquad\text{ for all }\varphi\in H^s(\Gamma).
$$
The formula \eqref{dLapBm1} means that the first derivative at $r=0$ in the direction $\xi\in\mathscr{C}^{\infty}(\Gamma,\R^3)$ of $r\mapsto\big(\mathcal{L}^{*}(r)\big)^{-1}f$ is the solution $v$ of 
 \begin{equation}\label{LapLip}
 \int_{\Gamma}\frac{\partial}{\partial r}\Big\{\Big(\mathcal{G}(r)u_{0}\cdot\mathcal{G}(r)\varphi\Big)\,J_{r}\Big\}[0,\xi]ds=-\int_{\Gamma}\nabla_{\Gamma}v\cdot\nabla_{\Gamma}\varphi\,ds,\;\text{ for all }\varphi\in H^s(\Gamma),
 \end{equation}
 with $u_{0}=\Delta_{\Gamma}^{-1}f$.\medskip
 
Now we have all the tools to establish the differentiability properties of the electromagnetic boundary integral operators and then of the solution to the dielectric scattering problem.

 \section{Shape derivatives of the solution of the  dielectric  problem}\label{ShapeSol}
For the shape-dependent integral operators we now use the following simplified notation
$$
\begin{aligned}
\Psi_{E_{\kappa}}(r)&=\Psi_{E_{\kappa}}^r\Pp_{r}^{-1},
   &\Psi_{M_{\kappa}}(r)=\Psi_{M_{\kappa}}^r\Pp_{r}^{-1},\ \;
\\
C_{\kappa}(r)&=\Pp_{r}C^r_{{\kappa}}\Pp_{r}^{-1},
  &\ \;M_{\kappa}(r)=\Pp_{r}M^r_{{\kappa}}\Pp_{r}^{-1}.
\end{aligned}
$$
In the following we use the results of the preceding paper \cite{CostabelLeLouer} about the G\^ateaux differentiability of potentials and boundary integral operators with pseudo-homogeneous kernels.
\begin{theorem}\label{thpsi}
 The mappings  
 $$
 \begin{array}{rcl}B^{\infty}(0,\varepsilon)&\rightarrow& \mathscr{L}(\TT\HH^{-\frac{1}{2}}(\Div_{\Gamma},\Gamma), \HH(\Rot,K_{p}))\\r&\mapsto&\Psi_{E_{\kappa}}(r)\\r&\mapsto&\Psi_{M_{\kappa}}(r)\end{array}
 $$ 
 are infinitely G\^ateaux differentiable. The derivatives can be written in explicit form by differentiating the kernels of the operators $\Psi_{E_{\kappa}}^r$ and $\Psi_{M_{\kappa}}^r$, see  \cite[Theorem 4.7]{CostabelLeLouer}, and by using the formulas for the derivatives of the surface differential operators given in Section~\ref{SurfDiffOp}.
The first derivatives at $r=0$ can be extended to bounded linear operators from $\TT\HH^{\frac{1}{2}}(\Div_{\Gamma},\Gamma)$ to $\HH(\Rot,\Omega)$ and to $\HH_{loc}(\Rot,\overline{\Omega^c})$.
Given  $\jj\in\TT\HH^{\frac{1}{2}}(\Div_{\Gamma},\Gamma)$, the potentials $d\Psi_{E_{\kappa}}[0,\xi]\jj$ and $d\Psi_{M_{\kappa}}[0,\xi]\jj$ satisfy the Maxwell equations 
$$
\Rot\Rot\uu-\kappa^2\uu=0
$$
in $\Omega$ and $\Omega^c$, and the Silver-M\"uller radiation condition. 
\end{theorem}

\begin{proof}
Let $\jj\in\TT\HH^{-\frac{1}{2}}(\Div_{\Gamma},\Gamma)$ and let $\jj=\nabla_{\Gamma}p+\Rot_{\Gamma}q$ be its Helmholtz decomposition. Recall that $\Psi_{E_{\kappa}}(r)\jj$ and $\Psi_{M_{\kappa}}(r)\jj$ can be written as: 
\begin{equation*}
\begin{split}
\Psi_{E_{\kappa}}(r)\,\jj&=\kappa \Psi^{r}_{\kappa}\tau_{r}^{-1}(\tau_{r}\mathbf{P_{r}}^{-1}\jj)-\dfrac{1}{\kappa}\nabla \Psi^{r}_{\kappa}\tau_{r}^{-1}\big(\tau_{r}\Delta_{\Gamma_{r}}(\tau_{r}^{-1}p)\big),\vspace{2mm} \\\Psi_{E_{\kappa}}(r)\,\jj&=\rot\Psi_{\kappa}^{r}\tau_{r}^{-1}(\tau_{r}\mathbf{P_{r}}^{-1}\jj).
\end{split}
\end{equation*}
By composition of differentiable mappings, we deduce that $r\mapsto\Psi_{E_{\kappa}}(r)$ and $r\mapsto~\Psi_{M_{\kappa}}(r)$ are infinitely G\^ateaux differentiable far from the boundary and that their first derivatives are continuous from $\TT\HH^{\frac{1}{2}}(\Div_{\Gamma},\Gamma)$ to $\LL^2(\Omega)\cup \LL^2_{\loc}(\overline{\Omega^c})$. Recall that we have
$$
\Rot\Psi_{E_{\kappa}}(r)\jj=\kappa\Psi_{M_{\kappa}}(r)\jj\text{ and }\Rot\Psi_{M_{\kappa}}(r)\jj=\kappa\Psi_{E_{\kappa}}(r)\jj.
$$ 
Far from the boundary  we can invert the differentiation with respect to $x$ and the derivation with respect to $r$, which gives 
$$
\Rot d\Psi_{E_{\kappa}}[0,\xi]\jj=\kappa\, d\Psi_{M_{\kappa}}[0,\xi]\jj\text{ and }\Rot d\Psi_{M_{\kappa}}[0,\xi]\jj=\kappa\, d\Psi_{E_{\kappa}}[0,\xi]\jj.
$$
It follows that $d\Psi_{E_{\kappa}}[0,\xi]\jj$ and $ d\Psi_{M_{\kappa}}[0,\xi]\jj$ are in $\HH(\Rot,\Omega)\cup\HH_{\loc}(\Rot,\overline{\Omega^c})$  and that  they satisfy the Maxwell equations and the Silver-M\"uller condition.
\end{proof}

 We recall from Section~\ref{Helmholtzdec} that with the notation of Section~\ref{SurfDiffOp} the operator $C_{\kappa}(r)$ admits the following representation
\begin{equation}
C_{\kappa}(r)\,\jj=\label{C}\mathbf{P_{r}}C^{r}_{\kappa}\mathbf{P_{r}}^{-1}\jj=\nabla_{\Gamma}P(r)+\Rot_{\Gamma}Q(r),
\end{equation} 
where
\begin{equation*}
\begin{array}{lcl}P(r)&=&-\kappa\;(\mathcal{L}^{*}(r))^{-1}\mathcal{R}^{*}(r)\left(\tau_{r}V^{r}_{\kappa}\tau_{r}^{-1}\right)\left[\mathcal{G}(r)p+\boldsymbol{\mathcal{R}}(r)q\right]\end{array}
\end{equation*}
and
\begin{equation*}
\begin{array}{lcl}Q(r)&=&-\kappa\;(\mathcal{L}^{*}(r))^{-1}\mathcal{D}^{*}(r)\pi(r)\tau_{r}^{-1}\left(\tau_{r}V^{r}_{\kappa}\tau_{r}^{-1}\right)\left[\mathcal{G}(r)p+\boldsymbol{\mathcal{R}}(r)q\right]
\vspace{3mm}\\
&&+\dfrac{1}{\kappa}\left(\tau_{r}V^{r}_{\kappa}\tau_{r}^{-1}\right)\left(\tau_{r}\Delta_{\Gamma_{r}}(\tau_{r}^{-1}p)\right).\end{array}
\end{equation*}

Let $\jj\in \TT\HH^{s}(\Div_{\Gamma},\Gamma)$ and let $\jj=\nabla_{\Gamma}\;p+\Rot_{\Gamma}\;q$ be its Helmholtz decomposition. We  want to derive 
$$
\begin{array}{ll}\Pp_{r}C^{r}_{\kappa}\Pp_{r}^{-1}\jj&=\Pp_{r}C^{r}_{\kappa}(\nabla_{\Gamma_{r}}\tau_{r}^{-1}p+\Rot_{\Gamma_{r}}\tau_{r}^{-1}q)\\&=\Pp_{r}(\nabla_{\Gamma_{r}}P_{r}+\Rot_{\Gamma_{r}}Q_{r})\\&=\nabla_{\Gamma}P(r)+\Rot_{\Gamma}Q(r).\end{array}
$$ 
We find
$$
dC_{\kappa}[0,\xi]\jj=\nabla_{\Gamma}dP[0,\xi]+\Rot_{\Gamma}dQ[0,\xi].
$$
Thus the derivative with respect to $r$  of $\mathbf{P_{r}}C^{r}_{\kappa}\mathbf{P_{r}}^{-1}\jj$ is given by the derivatives of the functions $P(r)=\tau_{r}(P_{r})$ and of $Q(r)=\tau_{r}(Q_{r})$.\newline
We also note that for an $r$-dependent vector function $f(r)$ on $\Gamma$ there holds 
$$
d \{\pi(r)\tau_{r}^{-1}f(r)\}[0,\xi]=\pi(0) df[0,\xi].
$$

By composition of infinitely differentiable mappings we obtain the following theorem.

\begin{theorem} 
The mapping: 
$$
\begin{array}{lcl}B^{\infty}(0,\varepsilon)&\rightarrow &\mathscr{L}\left(\TT\HH^{-\frac{1}{2}}(\Div_{\Gamma},\Gamma),\TT\HH^{-\frac{1}{2}}(\Div_{\Gamma},\Gamma)\right)\\r&\mapsto&\Pp_{r}C^{r}_{\kappa}\Pp_{r}^{-1}\end{array}
$$ 
is infinitely G\^ateaux differentiable.  The derivatives can be written in explicit form by differentiating the kernel of the operator $C_{\kappa}^r$, see  \cite[Corollary 4.5]{CostabelLeLouer}, and by using the formulas for the derivatives of the surface differential operators given in Section~\ref{SurfDiffOp}.
\end{theorem}

Similarly, recall that the operator $\Pp_{r}M^{r}_{\kappa}\mathbf{P_{r}}^{-1}$ admits the following representation :
\begin{equation*}\label{M}
\mathbf{P_{r}}M^{r}_{\kappa}\mathbf{P_{r}}^{-1}\jj=\nabla_{\Gamma}P'(r)+\Rot_{\Gamma}Q'(r),
\end{equation*}
where
\begin{equation*}\label{P'}
\begin{array}{rl}P'(r)=&(\mathcal{L}^{*}(r))^{-1}(\kappa^2J_{r}\cdot\tau_{r}\nn_{r}\cdot (\tau_{r}V^r_{\kappa}\tau_{r}^{-1})\left[\mathcal{G}(r)p+\boldsymbol{\mathcal{R}}(r)q\right]\vspace{2mm}\\& +(\mathcal{L}^{*}(r))^{-1}(J_{r}\cdot\tau_{r}D_{\kappa}^{r}\tau_{r}^{-1})(\tau_{r}\Delta_{\Gamma_{r}}(\tau_{r}^{-1}p))
\end{array}
\end{equation*}
and
\begin{equation*}
\begin{array}{ll}
Q'(r)=(\mathcal{L}^{*}(r))^{-1}\mathcal{R}^{*}(r)(\tau_{r}(B_{\kappa}^{r}-D_{\kappa}^{r})\tau_{r}^{-1})\left[\mathcal{G}(r)p+\boldsymbol{\mathcal{R}}(r)q\right]
\end{array}
\end{equation*}
with
\begin{equation*}
\begin{aligned}
 \tau_{r}B_{k}^r\Pp_{r}^{-1}\jj &=\tau_{r}\left\{\displaystyle{\int_{\Gamma_{r}}\nabla G(\kappa,|\cdot-y_{r}|)\left(\nn_{r}(\,\cdot\,)\cdot(\nabla_{\Gamma_{r}}\tau_{r}^{-1} p)(y_{r})\right)ds(y_{r})}\right. 
\\
 &\quad +\left.\displaystyle{\int_{\Gamma_{r}}\nabla G(\kappa,|\cdot-y_{r}|)\left(\nn_{r}(\,\cdot\,)\cdot(\Rot_{\Gamma_{r}}\tau_{r}^{-1} q)(y_{r})\right)ds(y_{r})}\Big\}\right\}.
\end{aligned}
\end{equation*}

\begin{theorem}
The mapping:
$$
\begin{array}{lcl}B^{\infty}(0,\varepsilon)&\rightarrow& \mathscr{L}\left(\TT\HH^{-\frac{1}{2}}(\Div_{\Gamma},\Gamma),\TT\HH^{-\frac{1}{2}}(\Div_{\Gamma},\Gamma)\right)\\r&\mapsto&\Pp_{r}M_{\kappa}\Pp_{r}^{-1}\end{array}
$$ 
is infinitely G\^ateaux differentiable. The G\^ateaux derivatives have the same regularity as $M_{\kappa}$, so that they are compact operators in $\TT\HH^{-\frac{1}{2}}(\Div_{\Gamma},\Gamma)$.  The derivatives can be written in explicit form by differentiating the kernel of the operators $M_{\kappa}^r$, see  \cite[Corollary 4.5]{CostabelLeLouer}, and by using the formulas for the derivatives of the surface differential operators given in Section~\ref{SurfDiffOp}.  
\end{theorem}

\begin{proof} 
The differentiability of the double layer boundary integral operator is established in \cite[Example 4.10]{CostabelLeLouer}.  It remains to prove the infinite G\^ateaux differentiability of the mapping
$$
\begin{array}{lcl}B_{\delta}&\rightarrow& \mathscr{L}\left(\TT\HH^{-\frac{1}{2}}(\Div_{\Gamma},\Gamma),\HH^{\frac{1}{2}}(\Gamma)\right)\\r&\mapsto&\tau_{r}B^r_{\kappa}\Pp_{r}^{-1}.\end{array}
$$
The function $(x,y-x)\mapsto \nabla G(\kappa,|x-y|)$ is pseudo-homogeneous of class 0. We then have to prove that for any fixed $(x,y)\in(\Gamma\times\Gamma)^*$ and any function $p\in H^{\frac{3}{2}}(\Gamma)$ the G\^ateaux derivatives of
$$
r\mapsto (\tau_{r}\nn_{r})(x)\cdot\left(\tau_{r}\nabla_{\Gamma_{r}}\tau_{r}^{-1}p\right)(y)
$$ behave as $|x-y|^2$ when $x-y$ tends to zero. To do so, either we write 
$$
(\tau_{r}\nn_{r})(x)\cdot\left(\tau_{r}\nabla_{\Gamma_{r}}\tau_{r}^{-1}p\right)(y)=\left((\tau_{r}\nn_{r})(x)-(\tau_{r}\nn_{r})(y)\right)\cdot\left(\tau_{r}\nabla_{\Gamma_{r}}\tau_{r}^{-1}p\right)(y)
$$
or we use Lemmas \ref{N} and \ref{nabla} and straighforward computations.
\end{proof}

\begin{theorem} \label{shapederiv} 
Assume that $\EE^{inc}\in\HH^1_{\loc}(\Rot,\R^3)$ and that the mappings
$$
\begin{array}{rcl}B^{\infty}(0,\varepsilon)&\rightarrow&\TT\HH^{-\frac{1}{2}}(\Div_{\Gamma},\Gamma)\\
 r&\mapsto&\Pp_{r}\left(\nn_{r}\wedge\EE^{inc}_{|\Gamma_{r}}\right)\\
 r&\mapsto& \Pp_{r}\left(\nn_{r}\wedge\left(\Rot\EE^{inc}\right)_{|{\Gamma_{r}}}\right)\end{array}
$$ 
are G\^ateaux differentiable at $r=0$. Then the mapping from $r\in B^{\infty}(0,\varepsilon)$ to the solution $\mathscr{E}(r)=\EE(\Omega_{r})\in\HH(\Rot,\Omega_{r})\cup\HH_{\loc}(\Rot,\overline{\Omega^c})$ of the scattering problem by the obstacle $\Omega_{r}$ is G\^ateaux differentiable at $r=0$.
\end{theorem}

\begin{proof} 
We use the integral equation method described in Theorem~\ref{thT}. 
Let $\jj$ be the solution of the integral equation \eqref{SS}. 
By composition of infinitely differentiable mappings we see that 
$$
\begin{array}{lcl}B^{\infty}(0,\varepsilon)&\rightarrow &\mathscr{L}\left(\TT\HH^{-\frac{1}{2}}(\Div_{\Gamma},\Gamma),\TT\HH^{-\frac{1}{2}}(\Div_{\Gamma},\Gamma)\right)\\r&\mapsto&\SS(r)=\Pp_{r}\SS^r\Pp_{r}^{-1}\end{array}
$$
is $\mathscr{C}^{\infty}$-G\^ateaux differentiable. Then with \eqref{u2c} and \eqref{SS} we get for the exterior field $\mathscr{E}^s$
\begin{equation*}
 \begin{split}
 d \mathscr{E}^s[0,\xi]=\;&\left(-d\Psi_{E_{\kappa_{e}}}[0,\xi]-i\eta d\Psi_{M_{\kappa_{e}}}[0,\xi]C^{*}_{0}-i\eta\Psi_{M_{\kappa_{e}}}d C^{*}_{0}[0,\xi]\right)\jj
 \\
 &+(-\Psi_{E_{\kappa_{e}}}-i\eta\Psi_{M_{\kappa_{e}}}C^*_{0})\SS^{-1}\big(-d \SS[0,\xi]\,\jj\big)
 \\
 +(-&\Psi_{E_{\kappa_{e}}}\!\!\!-i\eta\Psi_{M_{\kappa_{e}}}C^*_{0})\SS^{-1}\left(-\rho d M_{\kappa_{i}}[0,\xi]\gamma_{D}\EE^{inc}- d C_{\kappa_{i}}[0,\xi]\gamma_{N_{\kappa_{e}}}\EE^{inc}\right)
 \\
 &\!\!\!+(-\Psi_{E_{\kappa_{e}}}-i\eta\Psi_{M_{\kappa_{e}}}C^*_{0})\SS^{-1}\left(-\rho\Big(\tfrac{1}{2}+M_{\kappa_{i}}\right)d\left\{ \Pp_{r}\gamma_{D}^{r}\EE^{inc}\right\}[0,\xi]\Big)
 \\
 &+(-\Psi_{E_{\kappa_{e}}}-i\eta\Psi_{M_{\kappa_{e}}}C^*_{0})\SS^{-1}\left(-C_{\kappa_{i}}d\left\{ \Pp_{r}\gamma_{N_{\kappa_{e}}}^{r}\EE^{inc}\right\}[0,\xi]\right).
 \end{split}
\end{equation*}
We know that $\jj\in\TT\HH^{\frac{1}{2}}(\Div_{\Gamma},\Gamma)$, so that the first terms on the right-hand side are in $\HH_{\loc}(\Rot,\overline{\Omega^c})$, and the hypotheses guarantee that the last two terms are in $\HH_{\loc}(\Rot,\overline{\Omega^c})$.
For the interior field we write 
\begin{equation*}
\begin{split}
 d \mathscr{E}^{i}[0,\xi]=\;&-\frac{1}{\rho}d \Psi_{E_{\kappa_{i}}}[0,\xi]\gamma_{N_{\kappa_{e}}}^c\left(\EE^{s}+\EE^{inc}\right)-d\Psi_{M_{\kappa_{i}}}[0,\xi]\gamma_{D}^c\left(\EE^{s}+\EE^{inc}\right)
 \\
  &-\frac{1}{\rho}\Psi_{E_{\kappa_{i}}}d\hspace{-.5mm}\left\{\Pp_{r}\gamma_{N_{\kappa_{e}}}^{c,r}\left(\mathscr{E}^s(r)+\EE^{inc}\right)\right\}[0,\xi]\\
  &\qquad\qquad-\Psi_{M_{\kappa_{i}}}d\hspace{-.5mm}\left\{\Pp_{r}\gamma_{D}^{c,r}\left(\mathscr{E}^s(r)+\EE^{inc}\right)\right\}\hspace{-1mm}[0,\xi].
 \end{split}
\end{equation*}
The hypotheses guarantee that $\gamma_{N_{\kappa_{e}}}^c\left(\EE^{s}+\EE^{inc}\right)$ and $\gamma_{D}^c\left(\EE^{s}+\EE^{inc}\right)$ are in $\TT\HH^{\frac{1}{2}}(\Div_{\Gamma},\Gamma)$, which implies that the first two terms are in $\HH(\Rot,\Omega)$, and  that the last two terms are in $\HH(\Rot,\Omega)$.
\end{proof}

\begin{theorem}
 The mapping sending $r\in B^{\infty}(0,\varepsilon)$ to  the far field pattern $\EE^{\infty}(\Omega_{r})\in\TT\mathscr{C}^{\infty}(S^2)$ of the solution to the scattering problem by the obstacle $\Omega_{r}$ is $\mathscr{C}^{\infty}$-G\^ateaux differentiable.
\end{theorem}
\begin{proof}
The mapping $B^{\infty}(0,\varepsilon)\ni r\mapsto\left\{(\hat{x},y)\mapsto e^{i\kappa\hat{x}\cdot(y+r(y))}\right\}\in \mathscr{C}^{\infty}(S^2\times\Gamma)$ is $\mathscr{C}^{\infty}$-G\^ateaux differentiable and  the derivatives  define smooth kernels. By the linearity of the integral we deduce that the boundary--to--far--field operators 
$$
\begin{array}{lcl}
B_{\delta}&\rightarrow&\mathscr{L}(\TT\HH^s(\Div_{\Gamma},\Gamma),\mathscr{C}^{\infty}(S^2))
\\
r&\mapsto&\Psi^{\infty}_{E_{\kappa}}(r)=\Psi^{\infty,r}_{E_{\kappa}}\tau_{r}^{-1}
\\
r&\mapsto&\Psi^{\infty}_{M_{\kappa}}(r)=\Psi_{M_{\kappa}}^{\infty,r}\tau_{r}^{-1}
\end{array}
$$ 
are $\mathscr{C}^{\infty}$-G\^ateaux differentiable.  For  $\jj\in~\TT\HH^s(\Div_{\Gamma},\Gamma)$ we have:
$$
d\Psi_{E_{\kappa}}^{\infty}[0,\xi]\jj(\hat{x})=i\kappa\hat{x}\wedge\left(\int_{\Gamma}e^{-i\kappa\hat{x}\cdot y}\big(\Div_{\Gamma}\xi(y)-i\kappa\hat{x}\cdot\xi(y)\big)\jj(y)ds(y)\right)\wedge\hat{x},
$$
and
$$
d\Psi_{M_{\kappa}}^{\infty}[0,\xi]\jj(\hat{x})=\kappa\hat{x}\wedge\left(\int_{\Gamma}e^{-i\kappa\hat{x}\cdot y}\big(\Div_{\Gamma}\xi(y)-i\kappa\hat{x}\cdot\xi(y)\big)\jj(y)ds(y)\right).
$$
We conclude by using the integral representation of $\EE^{\infty}(\Omega_{r})$ and previous theorems.
 \end{proof}
\subsection{Characterization of the first derivative}

The following theorem gives a  caracterization of the first G\^ateaux derivative  of $r\mapsto\mathscr{E}(r)$ at $r=0$.

\begin{theorem}
\label{T:firstder}
Under the hypotheses of Theorem {\rm \ref{shapederiv}}, the first derivative of the solution $\mathscr{E}(r)$ of the dielectric scattering problem at $r=0$ in the direction $\xi\in\mathscr{C}^{\infty}(\Gamma,\R^3)$  solves the following transmission problem :
\begin{equation}
\left\{\begin{aligned}
\Rot\Rot d \mathscr{E}^{i}[0,\xi]-\kappa_{i}^2 d\mathscr{E}^{i}[0,\xi]=0
\\
\Rot\Rot d \mathscr{E}^{s}[0,\xi]-\kappa_{e}^2 d\mathscr{E}^{s}[0,\xi]=0
\end{aligned}\right.
\end{equation} 
with the interface conditions 
\begin{equation}
\label{E:firstder}
\left\{\begin{aligned}
\nn\wedge d\mathscr{E}^{i}[0,\xi]-\nn\wedge d\mathscr{E}^{s}[0,\xi]=g_{D}
\\
\mu_{i}^{-1}\nn\wedge\Rot d\mathscr{E}^{i}[0,\xi]-\mu_{e}^{-1}\nn\wedge\Rot d\mathscr{E}^{s}[0,\xi]=g_{N},
\end{aligned}\right.
\end{equation}
where with the solution $(\EE^{i},\EE^{s})$ of the scattering problem, 
\begin{equation*}
\begin{split}
g_{D}=&-\left(\xi\cdot\nn\right)\Big(\nn\wedge\Rot\EE^{i}-\nn\wedge\Rot(\EE^s+\EE^{inc})\Big)\wedge\nn
\\
&+\Rot_{\Gamma}\Big((\xi\cdot\nn)\big(\nn\cdot\EE^{i}-\nn\cdot(\EE^s+\EE^{inc})\big)\Big),
\end{split}
\end{equation*}
 and 
\begin{equation*}
\begin{split}
g_{N}=&-\left(\xi\cdot\nn\right)\left(\dfrac{\kappa_{i}^2}{\mu_{i}}\nn\wedge\EE^{i}-\dfrac{\kappa_{e}^2}{\mu_{e}}\nn\wedge(\EE^s+\EE^{inc})\right)\wedge\nn
\\
&+\Rot_{\Gamma}\Big((\xi\cdot\nn)\;\left(\mu_{i}^{-1}\rot_{\Gamma}\EE^{i}-\mu_{e}^{-1}\rot_{\Gamma}(\EE^s+\EE^{inc})\right)\Big),
\end{split}
\end{equation*} 
and $d\mathscr{E}^{s}[0,\xi]$ satisfies the Silver-M\"uller radiation condition.
\end{theorem}

\begin{proof}
 We have shown in the previous paragraph that the potential operators and their G\^ateaux derivatives satisfy the Maxwell equations and the Silver-M\"uller radiation condition. 
 It remains to compute the boundary conditions. They can be obtained from the integral representation, but this is rather tedious. A simpler way consists in deriving for a fixed $x\in\Gamma$ the expression 
\begin{equation}\label{bvder}
\nn_{r}(x+r(x))\wedge\left(\mathscr{E}^i(r)(x+r(x))-\mathscr{E}^s(r)(x+r(x))-\EE^{inc}(x+r(x))\right)=0.
\end{equation} 
This gives
$$
\begin{aligned}
0=\;&d\mathcal{N}[0,\xi](x)\wedge\left(\EE^{i}(x)-\EE^s(x)-\EE^{inc}(x)\right)
\\
&+\nn(x)\wedge\left(d\mathscr{E}^i[0,\xi](x)-d\mathscr{E}^s[0,\xi](x)\right)
\\
&+\nn\wedge\left(\xi(x)\cdot\nabla\left(\EE^{i}-\EE^s-\EE^{inc}\right)\right).
\end{aligned}
$$
We now use the explicit form of the shape derivatives of the normal vector given in Lemma~\ref{N} : 
$d\mathcal{N}[0,\xi]=-\left[\nabla_{\Gamma}\xi\right]\nn$, and the formula
$
\nabla u=\nabla_{\Gamma}u+\left(\frac{\partial u}{\partial \nn}\right)\nn.
$
We obtain
\begin{multline*}
\nn(x)\wedge\left(d\mathscr{E}^i[0,\xi](x)-d\mathscr{E}^s[0,\xi](x)\right)=\\
\left[\nabla_{\Gamma}\xi\right]\nn\wedge\left(\EE^{i}(x)-\EE^s(x)-\EE^{inc}(x)\right)
\\
-\nn\wedge\left(\xi(x)\cdot\nabla_{\Gamma}\left(\EE^{i}(x)-\EE^s(x)-\EE^{inc}(x)\right)\right)
\\
-(\xi\cdot\nn)\nn\wedge\dfrac{\partial}{\partial\nn}\left(\EE^{i}(x)-\EE^s(x)-\EE^{inc}(x)\right).
\end{multline*}
 Since the tangential component of $\EE^{i}-\EE^s-\EE^{inc}$ vanishes, we have
\begin{multline*}
\left(\xi(x)\cdot\nabla_{\Gamma}\left(\EE^{i}(x)-\EE^s(x)-\EE^{inc}(x)\right)\right)\\
=\left(\nn\cdot\left(\EE^{i}(x)-\EE^s(x)-\EE^{inc}(x)\right)\right)\left(\left[\transposee{\nabla_{\Gamma}\nn}\right]\xi\right)
\end{multline*}
and 
\begin{multline*}
\left(\left[\nabla_{\Gamma}\xi\right]\nn\right)\wedge\left(\EE^{i}(x)-\EE^s(x)-\EE^{inc}(x)\right)\\
=\left(\left[\nabla_{\Gamma}\xi\right]\nn\right)\wedge\nn\left(\EE^{i}(x)-\EE^s(x)-\EE^{inc}(x)\right)\cdot\nn.
\end{multline*}
Since we are on a regular surface, we have $\nabla_{\Gamma}\nn=\transposee{\nabla_{\Gamma}\nn}$ and   
$$
 \left(\left[\nabla_{\Gamma}\xi\right]\nn\right)\wedge\nn-\nn\wedge\left(\left[\transposee{\nabla_{\Gamma}\nn}\right]\xi\right)=\Rot_{\Gamma}\left(\xi\cdot\nn\right).
$$
Using the expansion (see \cite[p.\ 75]{Nedelec})
$$
\Rot\uu=\left(\rot_{\Gamma}\uu\right)\nn+\Rot_{\Gamma}\left(\uu_{\Gamma}\cdot\nn\right)-\left([\nabla_{\Gamma}\nn]\uu\right)\wedge\nn-\left(\frac{\partial\uu}{\partial\nn}\right)\wedge\nn
$$ 
we obtain that
$$
\begin{aligned}
-\nn\wedge\left(\gamma_{n}\EE^{i}-\gamma_{n}^c\big(\EE^s+\EE^{inc}\big)\right)&=-\nn\wedge\left(\gamma\Rot\EE^{i}-\gamma^c\Rot\big(\EE^s+\EE^{inc}\big)\right)\wedge\nn
\\
&\qquad+\Rot_{\Gamma}\left(\nn\cdot\left(\gamma\EE^{i}-\gamma^c\big(\EE^s+\EE^{inc}\big)\right)\right).
\end{aligned}
$$
Here we used that the curvature operator $[\nabla_{\Gamma}\nn]$ acts on the tangential component of vector fields, so that 
$$
[\nabla_{\Gamma}\nn]\left(\gamma\EE^{i}-\gamma^c\big(\EE^s+\EE^{inc}\big)\right)=0.
$$
Thus we have
\begin{equation*}
\begin{split}
g_{D}=&-\left(\xi\cdot\nn\right)\nn\wedge\left(\gamma\Rot\EE^{i}-\gamma^c\Rot\big(\EE^s+\EE^{inc}\big)\right)\wedge\nn
\\
&+\Rot_{\Gamma}\left((\xi\cdot\nn)\big(\nn\cdot\gamma\EE^{i}-\nn\cdot\gamma^c(\EE^s+\EE^{inc})\big)\right).
\end{split}
\end{equation*}
To obtain the second transmission condition, we use similar computations with the electric field $\EE$ replaced by the magnetic field $\dfrac{1}{i\omega\mu}\Rot\EE$.
\end{proof}

\section{Perspectives: Non-smooth boundaries}

We have presented a complete differentiability analysis of the electromagnetic integral operators with respect to smooth deformations of a smooth boundary  in the framework of Sobolev spaces. Using the boundary integral equation approach we have established that the far-field pattern of the  dielectric scattering problem is infinitely differentiable with respect to the deformations and we gave a characterization of the first derivative as the far-field pattern of a new transmission problem.

In the case of a non-smooth boundary --- a polyhedral or more generally a Lipschitz boundary --- the formulas determining the first derivative given in Theorem~\ref{T:firstder} are problematic. The normal vector field $\nn$ will have discontinuities, and the factor $\xi\cdot\nn$ and vector product with $\nn$ that appear in the right-hand side of \eqref{E:firstder} may not be well defined in the energy trace spaces. It is, however, known for the acoustic case that the far field is infinitely shape differentiable for non-smooth boundaries, too, see \cite{Hohage} for a proof via the implicit function theorem. 

Our procedure of using a boundary integral representation gives an alternative way of characterizing the shape derivatives of the solution of the dielectric scattering problem and of its far field. We do not require the computation of the  boundary traces of the solution and taking tangential derivatives and multiplication by possibly discontinuous factors. Instead we determine the shape derivatives of the boundary integral operators. While the study of G\^ateaux differentiability of boundary integral operators, for the case of smooth deformations of a Lipschitz domain, is still an open problem that will require further work, our approach via Helmholtz decompositions seems to be a promising starting point for tackling this question. Let us briefly indicate why we think this is so.

We consider the case where the boundary $\Gamma$ is merely Lipschitz, but the deformation is defined by a vector field $\xi$ that is smooth (at least $\mathscr{C}^1$) in a neighborhood of $\Gamma$. Note that the reduction to purely normal displacements that is often used for studying shape optimization problems for smooth boundaries does not make sense here, as soon as there are corners present. In this situation, many of the ingredients of our toolbox are still available. Here are some of them.

First, the change of variables mapping $\tau_{r}$ still defines an isomorphism between $H^{\frac{1}{2}}(\Gamma_{r})$ and $H^{\frac{1}{2}}(\Gamma)$. By duality, we see that the mapping
$u_{r}\mapsto J_{r}\tau_{r}u_{r}$
defines an isomorphism between $H^{-\frac{1}{2}}(\Gamma_{r})$ and $H^{-\frac{1}{2}}(\Gamma)$. When we want to transport the energy space $\TT\HH^{-\frac{1}{2}}(\Div_{\Gamma_{r}},\Gamma_{r})$, we can still use Helmholtz decomposition. Namely, the following result is known, see \cite{BuffaCiarlet,BuffaCiarlet2,BuffaCostabelSchwab,BuffaCostabelSheen,BuffaHiptmairPetersdorffSchwab}.

\begin{lemma} 
\label{HelmLip}
Assume that  $\Gamma$ is a simply connected closed Lipschitz surface.  The Hilbert space $\TT\HH_{\|}^{-\frac{1}{2}}(\Div_{\Gamma},\Gamma)$ admits the following Helmholtz decomposition:
$$ 
\TT\HH_{\|}^{-\frac{1}{2}}(\Div_{\Gamma},\Gamma)= 
 \nabla_{\Gamma}\mathcal{H}(\Gamma)\oplus 
 {\Rot}_{\Gamma}\,H^{\frac{1}{2}}(\Gamma).
$$
 where
$$
\mathcal{H}(\Gamma)=\{u\in H^1(\Gamma)\;:\;\Delta_{\Gamma}u\in H^{-\frac{1}{2}}(\Gamma)\}.
$$
\end{lemma}
The notation $\TT\HH_{\|}^{-\frac{1}{2}}(\Div_{\Gamma},\Gamma)$ recalls the fact that special care has to be taken for the definition of the energy space. 

A natural idea for the transport of the energy trace space is then, instead of \eqref{Prsmooth}, to define
\begin{equation*}
 \begin{array}{rrcl} 
  \Pp_{r}:&\TT\HH_{\|}^{-\frac{1}{2}}(\Div_{\Gamma_{r}},\Gamma_{r})&\longrightarrow& \TT\HH_{\|}^{-\frac{1}{2}}(\Div_{\Gamma},\Gamma)\\
  &\nabla_{\Gamma_{r}}\;p_{r}+\Rot_{\Gamma_{r}}\;q_{r}&\mapsto&\nabla_{\Gamma}\;\Delta_{\Gamma}^{-1}\left(J_{r}\big(\tau_{r}\Delta_{\Gamma_{r}}p_{r}\big)\right)+\Rot_{\Gamma}(\tau_{r}q_{r}).
 \end{array}
\end{equation*}
This is justified by the sequence of isomorphisms
$$
  \begin{array}{ccccccc}
   \mathcal{H}(\Gamma_{r})/\C&\displaystyle{
\stackrel{\Delta_{\Gamma_{\hspace{-.5mm}r}}\;}{\longrightarrow}}&H^{-\frac{1}{2}}_{*}(\Gamma_{r})&
{\longrightarrow}&H^{-\frac{1}{2}}_{*}(\Gamma)&\stackrel{\Delta_{\Gamma}^{-1}}{\longrightarrow}&
\mathcal{H}(\Gamma)/\C\\
   p_{r}&\mapsto&\Delta_{\Gamma_{r}}p_{r}&\mapsto&J_{r}(\tau_{r}\Delta_{\Gamma_{r}}p_{r})&\mapsto&\Delta_{\Gamma}^{-1}\big(J_{r}(\tau_{r}\Delta_{\Gamma_{r}}p_{r})\big).
  \end{array}
$$
The inverse of the transformation $\Pp_{r}$ is given by
\begin{equation*}
  \begin{array}{rrcl} 
  \Pp_{r}^{-1}:&\TT\HH_{\|}^{-\frac{1}{2}}(\Div_{\Gamma},\Gamma)&\longrightarrow& \TT\HH_{\|}^{-\frac{1}{2}}(\Div_{\Gamma_{r}},\Gamma_{r})\\
  &\nabla_{\Gamma}\;p+\Rot_{\Gamma}\;q&\mapsto&\nabla_{\Gamma_{r}}\;\tau_{r}^{-1}\big(\mathcal{L}^{*}(r)\big)^{-1}\Delta_{\Gamma}p+\Rot_{\Gamma_{r}}(\tau_{r}^{-1}q).
  \end{array}
\end{equation*}
In this situation it seems  to be more convenient  to rewrite the operators $\Pp_{r}C^{r}_{\kappa}\Pp_{r}^{-1}$ and $\Pp_{r}M^{r}_{\kappa}\Pp_{r}^{-1}$ as operators acting on the  the scalar fields 
$p^{*}=\Delta_{\Gamma}p\in H^{-\frac{1}{2}}_{*}(\Gamma)$ and 
$q\in H^{\frac{1}{2}}(\Gamma_{r})/\C$ instead of $p$ and $q$. For example, the operator $\Pp_{r}C^{r}_{\kappa}\Pp_{r}^{-1}$ is defined for 
$\jj=\nabla_{\Gamma}\Delta_{\Gamma}^{-1}p^{*}+\Rot_{\Gamma}q\in \TT\HH_{\|}^{-\frac{1}{2}}(\Div_{\Gamma},\Gamma)$ by
$$
\Pp_{r}C^{r}_{\kappa}\Pp_{r}^{-1}=\nabla_{\Gamma}\Delta_{\Gamma}^{-1}P^{*}(r)+\Rot_{\Gamma}Q(r),
$$ 
with
\begin{equation*}
\begin{array}{lcl}P(r)&=&-\kappa\;\mathcal{R}^{*}(r)\left(\tau_{r}V^{r}_{\kappa}\tau_{r}^{-1}\right)\left[\mathcal{G}(r)\big(\mathcal{L}^{*}(r)\big)^{-1}p^{*}+\boldsymbol{\mathcal{R}}(r)q\right]\end{array}
\end{equation*}
and
\begin{equation*}
\begin{array}{lcl}Q(r)&=&-\kappa\;(\mathcal{L}^{*}(r))^{-1}\mathcal{D}^{*}(r)\pi(r)\left(\tau_{r}V^{r}_{\kappa}\tau_{r}^{-1}\right)\left[\mathcal{G}(r)\big(\mathcal{L}^{*}(r)\big)^{-1}p^{*}+\boldsymbol{\mathcal{R}}(r)q\right]
\vspace{3mm}\\
&&+\dfrac{1}{\kappa}\left(\tau_{r}V^{r}_{\kappa}\tau_{r}^{-1}\right)\left(J_{r}^{-1}p^{*}\right).
\end{array}
\end{equation*}
Here we have used the same notation for the surface differential operators as introduced in Section~\ref{SurfDiffOp}. These formulas together with similar ones for the operator $\Pp_{r}M_{\kappa}\Pp_{r}^{-1}$ can now be the starting point for generalization of the analysis of shape differentiability of the Maxwell boundary integral operators to Lipschitz domains. We expect that the results for the differentiability of the surface differential operators and then also of the boundary integral operators will be similar to what we have obtained for the case of smooth domains. This is, however, far from trivial and will require further work.

%

\begin{thebibliography}{10}

\bibitem{BuffaCiarlet}
{\sc A.~Buffa and P.~Ciarlet, Jr.}, {\em On traces for functional spaces
  related to {M}axwell's equations. {I}. {A}n integration by parts formula in
  {L}ipschitz polyhedra}, Math. Methods Appl. Sci., 24 (2001), pp.~9--30.

\bibitem{BuffaCiarlet2}
\leavevmode\vrule height 2pt depth -1.6pt width 23pt, {\em On traces for
  functional spaces related to {M}axwell's equations. {II}. {H}odge
  decompositions on the boundary of {L}ipschitz polyhedra and applications},
  Math. Methods Appl. Sci., 24 (2001), pp.~31--48.

\bibitem{BuffaCostabelSchwab}
{\sc A.~Buffa, M.~Costabel, and C.~Schwab}, {\em Boundary element methods for
  {M}axwell's equations on non-smooth domains}, Numer. Math., 92 (2002),
  pp.~679--710.

\bibitem{BuffaCostabelSheen}
{\sc A.~Buffa, M.~Costabel, and D.~Sheen}, {\em On traces for {${\bf H}({\bf
  curl},\Omega)$} in {L}ipschitz domains}, J. Math. Anal. Appl., 276 (2002),
  pp.~845--867.

\bibitem{BuffaHiptmairPetersdorffSchwab}
{\sc A.~Buffa, R.~Hiptmair, T.~von Petersdorff, and C.~Schwab}, {\em Boundary
  element methods for {M}axwell transmission problems in {L}ipschitz domains},
  Numer. Math., 95 (2003), pp.~459--485.

\bibitem{ColtonKress}
{\sc D.~Colton and R.~Kress}, {\em Inverse acoustic and electromagnetic
  scattering theory}, vol.~93 of Applied Mathematical Sciences,
  Springer-Verlag, Berlin, second~ed., 1998.

\bibitem{MartC}
{\sc M.~Costabel}, {\em Boundary integral operators on {L}ipschitz domains:
  elementary results}, SIAM J. Math. Anal., 19 (1988), pp.~613--626.

\bibitem{CostabelLeLouer2}
{\sc M.~Costabel and F.~Le~Lou{\"e}r}, {\em On the {K}leinman-{M}artin integral
  equation method for the electromagnetic scattering problem by a dielectric
  body}, SIAM J. Appl. Math, 71 (2011), pp.~635--656.

\bibitem{CostabelLeLouer}
\leavevmode\vrule height 2pt depth -1.6pt width 23pt, {\em Shape derivatives of
  boundary integral operators in electromagnetic scattering. {P}art {I}: Shape
  differentiability of pseudo-homogeneous boundary integral operators},
  (2011).

\bibitem{delaBourdonnaye}
{\sc A.~de~La~Bourdonnaye}, {\em D\'ecomposition de {$H\sp {-1/2}\sb {\rm
  div}(\Gamma)$} et nature de l'op\'erateur de {S}teklov-{P}oincar\'e du
  probl\`eme ext\'erieur de l'\'electromagn\'etisme}, C. R. Acad. Sci. Paris
  S\'er. I Math., 316 (1993), pp.~369--372.

\bibitem{DelfourZolesio}
{\sc M.~C. Delfour and J.-P. Zol{\'e}sio}, {\em Shapes and geometries}, vol.~4
  of Advances in Design and Control, Society for Industrial and Applied
  Mathematics (SIAM), Philadelphia, PA, 2001.
\newblock Analysis, differential calculus, and optimization.

\bibitem{DelfourZolesio2}
{\sc M.~C. Delfour and J.-P. Zol{\'e}sio}, {\em Tangential calculus and shape
  derivatives}, vol.~216 of Lecture Notes in Pure and Appl. Math., Dekker, New
  York, 2001.

\bibitem{Hadamard}
{\sc J.~Hadamard}, {\em Sur quelques questions du calcul des variations}, Ann.
  Sci. \'Ecole Norm. Sup. (3), 24 (1907), pp.~203--231.

\bibitem{HadarKress}
{\sc H.~Haddar and R.~Kress}, {\em On the {F}r\'echet derivative for obstacle
  scattering with an impedance boundary condition}, SIAM J. Appl. Math., 65
  (2004), pp.~194--208 (electronic).

\bibitem{PierreHenrot}
{\sc A.~Henrot and M.~Pierre}, {\em Variation et optimisation de formes},
  vol.~48 of Math\'ematiques \& Applications (Berlin) [Mathematics \&
  Applications], Springer, Berlin, 2005.
\newblock Une analyse g{\'e}om{\'e}trique. [A geometric analysis].

\bibitem{Hettlich}
{\sc F.~Hettlich}, {\em Fr\'echet derivatives in inverse obstacle scattering},
  Inverse Problems, 11 (1995), pp.~371--382.

\bibitem{HettlichErra}
\leavevmode\vrule height 2pt depth -1.6pt width 23pt, {\em Erratum: ``{F}rechet
  derivatives in inverse obstacle scattering'' [{I}nverse {P}roblems {\bf 11}
  (1995), no. 2, 371--382; {MR}1324650 (95k:35217)]}, Inverse Problems, 14
  (1998), pp.~209--210.

\bibitem{HettlichRundell}
{\sc F.~Hettlich and W.~Rundell}, {\em A second degree method for nonlinear
  inverse problems}, SIAM J. Numer. Anal., 37 (2000), pp.~587--620
  (electronic).

\bibitem{Hohage}
{\sc T.~Hohage}, {\em Iterative Methods in Inverse Obstacle Scattering:
  Regularization Theory of Linear and Nonlinear Exponentially Ill-Posed
  Problems}, {PhD} in {N}umerical analysis, University of Linz, 1999.

\bibitem{HsiaoWendland}
{\sc G.~C. Hsiao and W.~L. Wendland}, {\em Boundary integral equations},
  vol.~164 of Applied Mathematical Sciences, Springer-Verlag, Berlin, 2008.

\bibitem{Kirsch}
{\sc A.~Kirsch}, {\em The domain derivative and two applications in inverse
  scattering theory}, Inverse Problems, 9 (1993), pp.~81--96.

\bibitem{Kress}
{\sc R.~Kress}, {\em Electromagnetic waves scattering : Scattering by
  obstacles}, Scattering,  (2001), pp.~191--210.
\newblock Pike, E. R. and Sabatier, P. C., eds., Academic Press, London.

\bibitem{KressPaivarinta}
{\sc R.~Kress and L.~P{\"a}iv{\"a}rinta}, {\em On the far field in obstacle
  scattering}, SIAM J. Appl. Math., 59 (1999), pp.~1413--1426 (electronic).

\bibitem{FLL}
{\sc F.~Le~Lou{\"e}r}, {\em Optimisation de formes d'antennes lentilles
  int{\'e}gr{\'e}es aux ondes millim{\'e}trique}, {PhD} in {N}umerical
  {A}nalysis, Universit{\'e} de Rennes 1, 2009.
\newblock \newline \texttt{http://tel.archives-ouvertes.fr/tel-00421863/fr/}.

\bibitem{LeugeringetalAMOptim11}
{\sc G.~Leugering, A.~A. Novotny, G.~P. Menzala, and J.~Soko{\l}owski}, {\em On
  shape optimization for an evolution coupled system}, Appl. Math. Optim. 64,
  (2011), pp.~441--466.

\bibitem{MartinOla}
{\sc P.~A. Martin and P.~Ola}, {\em Boundary integral equations for the
  scattering of electromagnetic waves by a homogeneous dielectric obstacle},
  Proc. Roy. Soc. Edinburgh Sect. A, 123 (1993), pp.~185--208.

\bibitem{Monk}
{\sc P.~Monk}, {\em Finite element methods for {M}axwell's equations},
  Numerical Mathematics and Scientific Computation, Oxford University Press,
  New York, 2003.

\bibitem{Nedelec}
{\sc J.-C. N{\'e}d{\'e}lec}, {\em Acoustic and electromagnetic equations},
  vol.~144 of Applied Mathematical Sciences, Springer-Verlag, New York, 2001.
\newblock Integral representations for harmonic problems.

\bibitem{Potthast2}
{\sc R.~Potthast}, {\em Fr\'echet differentiability of boundary integral
  operators in inverse acoustic scattering}, Inverse Problems, 10 (1994),
  pp.~431--447.

\bibitem{Potthast3}
\leavevmode\vrule height 2pt depth -1.6pt width 23pt, {\em Domain derivatives
  in electromagnetic scattering}, Math. Methods Appl. Sci., 19 (1996),
  pp.~1157--1175.

\bibitem{Potthast1}
\leavevmode\vrule height 2pt depth -1.6pt width 23pt, {\em Fr\'echet
  differentiability of the solution to the acoustic {N}eumann scattering
  problem with respect to the domain}, J. Inverse Ill-Posed Probl., 4 (1996),
  pp.~67--84.

\bibitem{Potthast4}
\leavevmode\vrule height 2pt depth -1.6pt width 23pt, {\em
  Fr{\'e}chet-Differenzierbarkeit von Randintegraloperatoren und
  Randwertproblemen zur Helmholtzgleichung und den zeitharmonischen
  Maxwellgleichungen}, {PhD} in {N}umerical {A}nalysis, G{\"o}ttingen, 1996.

\bibitem{Schwartz}
{\sc J.~T. Schwartz}, {\em Nonlinear functional analysis}, Gordon and Breach
  Science Publishers, New York, 1969.
\newblock Notes by H. Fattorini, R. Nirenberg and H. Porta, with an additional
  chapter by Hermann Karcher, Notes on Mathematics and its Applications.

\bibitem{Zolesio}
{\sc J.~Soko{\l}owski and J.-P. Zol{\'e}sio}, {\em Introduction to shape
  optimization}, vol.~16 of Springer Series in Computational Mathematics,
  Springer-Verlag, Berlin, 1992.
\newblock Shape sensitivity analysis.

\end{thebibliography}

\end{document}